\documentclass[11pt,
]{amsart}
\usepackage{
amssymb, graphicx
}
\usepackage{mcode}
\usepackage{array}
\usepackage{graphicx, caption, subcaption}
\usepackage{epstopdf}

\newtheorem{proposition}{Proposition}

\newtheorem{remark}{Remark}
\usepackage{algorithm}
\usepackage{algorithmic}
\date{\today}

\newcommand{\be}[1]{\begin{equation}\label{#1}}
\newcommand{\ee}{\end{equation}}

\title[Inversion of trace formulas]{Inversion of trace formulas for a Sturm-Liouville operator}
  \author[X. Xu]{Xiang Xu}
\address{School of Mathematical Sciences,
  Zhejiang University, Hangzhou, China
  (\tt{xxu@zju.edu.cn}).}
  \author[J. Zhai]{Jian Zhai}
\address{Institute for Advanced Study,
  The Hong Kong University of Science and Technology, Hong Kong, China
  (\tt{jian.zhai@outlook.com}).}

\begin{document}

\begin{abstract}
 This paper revisits the classical problem ``Can we hear the density of a string?", which can be formulated as an inverse spectral problem for a Sturm-Liouville operator. Based on inverting a sequence of trace formulas,  we propose a new numerical scheme to reconstruct the density. Numerical experiments are presented to verify the validity and effectiveness of the numerical scheme.
\end{abstract}
\maketitle
\section{Introduction}
The main purpose of this paper is to propose a novel numerical scheme for some inverse spectral problems. The scheme can be generalized, but for simplicity will be presented with an \textit{ad hoc} algorithm for the classical Sturm-Liouville eigenvalue problem:
\begin{equation}\label{eq1}
\begin{split}
&-\frac{\mathrm{d}^2u}{\mathrm{d}x^2}=\lambda\rho u,\quad\quad\quad x\in[0,1],\\
&u(0)=u(1)=0,
\end{split}
\end{equation}
where $\rho\in L^\infty(0,1)$ and $\rho>0$. Denote $\{\lambda_k(\rho)\}_{k=1}^\infty$ to be the generalized eigenvalues for the above problem. The \emph{inverse spectral problem} is to recover the density $\rho$ from those eigenvalues $\{\lambda_k\}_{k=1}^\infty$. To avoid the nonuniqueness, we assume $\rho(x)$ is even with respect to $x=\frac{1}{2}$, that is $\rho(x)=\rho(1-x)$. 

If $\rho\in C^2([0,1])$, then the \textit{Liouville transformation}
\[
\sigma(x)=\sqrt{\rho(x)},\quad f(x)=\rho(x)^{1/4},\quad L=\int_0^1\sigma(s)\mathrm{d}s,\quad t(x)=\frac{1}{L}\int_0^x\sigma(s)\mathrm{d}s,\quad v(t)=f(x(t))u(x(t))
\]
reduces the inverse problem to the problem of recovering $q(t)$ in
\begin{equation}\label{eq11}
\begin{split}
&-\frac{\mathrm{d}^2 v}{\mathrm{d}t^2}+qv=L^2\lambda v, \quad\quad t\in [0,1],\\
&v(0)=v(1)=0,
\end{split}
\end{equation}
from the eigenvalues $\{L^2\cdot\lambda_k(\rho)\}_{k=1}^\infty$, where
\begin{equation}\label{qrelation}
q(t)=L^2\frac{f(x)}{\rho(x)}\left(\frac{f'(x)}{f(x)^2}\right)'\Bigg\vert_{x=x(t)}.
\end{equation}

The real-valued number $L$ can be recovered from spectral data \cite{rundell1992reconstruction1}. And the uniqueness of $q$ in $(\ref{eq11})$ is guaranteed if $q$ is assumed to be symmetric about $x=\frac{1}{2}$. For more details, we refer to \cite{kirsch2011introduction} and the discussions therein. It is still notable that the problems for $(\ref{eq1})$ and $(\ref{eq11})$ might be slightly different. First, the recovery of $\rho(x)$ from $q(x)$ from the relation $(\ref{qrelation})$ might need some additional information. Second, we need some regularity assumption on $\rho(x)$ to do the \textit{Liouville transformation}. In this paper, we will directly deal with problem $(\ref{eq1})$, and hence only need to assume $\rho(x)$ is bounded, positive and even with respect to $x=\frac{1}{2}$ for the discussions below.

Numerical methods for one dimensional inverse spectral problems $(\ref{eq1})$ and $(\ref{eq11})$ abound. We refer to \cite{lowe1992recovery,rundell1992reconstruction1,rundell1992reconstruction} and the references therein.  However, most methods rely heavily on the one-dimensional nature of the problem. In this paper, we give a novel numerical method for the reconstruction of $\rho$ in $(\ref{eq1})$ based on trace formulas. Moreover, we believe that this method can be adapted to some other, and even higher dimensional inverse spectral problems.

We need Fr\'echet differentiability of the forward map if we intend to use some gradient-based algorithm for solving inverse problems. 
Assume we have an elliptic operator $\mathcal{M}(\rho)$ on a compact manifold (with or without boundary) with a parameter $\rho$, and we want to recover $\rho$ from the eigenvalues $\{\lambda_k(\rho)\}_{k=1}^\infty$ of $\mathcal{M}(\rho)$. The differentiability of the map
\begin{equation}\label{map}
\rho\rightarrow \{\lambda_k(\rho)\}_{k=1}^\infty,
\end{equation}
is a very delicate issue. Perturbation theory of eigenvalues and eigenfunctions has been studied by Kato \cite{kato2013perturbation}, but is quite inaccessible.
Although this is not a problem for the inverse spectral problem related to $\eqref{eq1}$ due to the fact that all eigenvalues are simple, and many methods do not directly invert this map because of other concerns, one needs to bear in mind that
it might be a serious issue for some higher dimensional or non-Hermitian problems.
In this article, we will give an inversion scheme based on a different map, arising from trace formulas, of which the Fr\'echet differentiability is well guaranteed.
For \eqref{eq1}, the map is given as follows
\begin{equation}\label{map1}
\rho\rightarrow \left\{\sum_{k=1}^\infty\lambda_k^{-s}\right\}_{s=1}^\infty.
\end{equation}
To be more precise, we will actually invert the following map, for the sake of numerical stability,
\begin{equation}\label{map11}
\rho\rightarrow \left\{\sum_{k=1}^\infty\mathcal{P}_n(\lambda_k^{-1})\right\}_{n=1}^\infty,
\end{equation}
where $\{\mathcal{P}_n\}_n^\infty$ is a sequence of carefully chosen polynomials.

The rest of the paper is organized as follows. In Section \ref{traceformula}, we derive a sequence of trace formulas that will be utilized for inversion. In Section \ref{algorithms}, the inversion scheme is presented  together with various implementation details. In Section \ref{numerical}, we present some numerical experiments and discuss the performance of the algorithm. In Section \ref{discussion}, we discuss the merits and limitations of the proposed algorithm.

\section{Trace formulas}\label{traceformula}
Our starting point for the new algorithm is the observation of the identity
\[
\sum_{k=1}^\infty \lambda_k^{-1}=\text{trace}(A\circ M_\rho)=\int_0^1x(x-1)\rho(x)\mathrm{d}x,
\]
which was mentioned in \cite{gesztesy2011damped}. The above identity gives an explicit relation between the density $\rho(x)$ and the eigenvalues. We will derive formulas for $\sum_{k=1}^\infty \lambda_k^{-s}$, $s\geq 2,\,3,\cdots,$ in this section.

We denote the Dirichlet Laplacian as $\Delta_D$. Define $A=(-\Delta_D)^{-1}\in\mathcal{L}(L^2(0,1),L^2(0,1))$ such that $v=Af$ satisfies
\begin{equation}
\begin{split}
&-\frac{\mathrm{d}^2v}{\mathrm{d}x^2}=f,\quad\quad\quad x\in[0,1],\\
&v(0)=v(1)=0.
\end{split}
\end{equation}
The operator $A$ is an integral operator, in the form
\[
(Af)(x)=\int_0^1g(x,y)f(y)\mathrm{d}y,
\]
with the associated Schwartz kernel
\[
g(x,y)=\begin{cases}x(1-y),\quad\quad 0\leq x\leq y\leq 1,\\
y(1-x),\quad\quad 0\leq y\leq x\leq 1.
\end{cases}
\]
It is clear that $g(\cdot\,,\,\cdot)$ is continuous on $[0,1]\times[0,1]$.

Denote the operator $T$ as $Tf=u$, where $u$ solves
\begin{equation}
\begin{split}
&-\frac{\mathrm{d}^2u}{\mathrm{d}x^2}=\rho f,\quad\quad\quad x\in[0,1],\\
&u(0)=u(1)=0.
\end{split}
\end{equation}
Then we have $T=A\circ M_\rho$, where $M_\rho\in\mathcal{L}(L^2(0,1),L^2(0,1))$ is the multiplication operator $M_\rho f=\rho f$, and the Schwartz kernel of $T$ is
\[
G(x,y)=g(x,y)\rho(y).
\]

Recall that a bounded linear operator $B$ over a separable Hilbert space $\mathcal{H}$ is said to be of trace class if for some orthonormal bases $\{e_k\}_{k}$ of $\mathcal{H}$, the sum
\[
\sum_{k}\langle (B^*B)^{1/2}e_k,e_k\rangle
\]
is finite. In this case, the trace of $B$ is given by
\[
\mathrm{trace} (B)=\sum_{k}\langle Be_k,e_k\rangle.
\]
Here $\mathrm{trace} (B)$ is independent of the choice of orthonormal basis, and the Lidskii's theorem states that if $\{\kappa_k(B)\}$ are the eigenvalues of $B$, then
\[
\mathrm{trace} (B)=\sum_{k}\kappa_k(B).
\]

Also remember that a bounded operator $B$ is called Hilbert-Schmidt if $B^*B$ is of trace class. We need the following propositions to continue. For more details, we refer to \cite{Lax}.
\begin{proposition}
If $B$ and $C$ are Hilbert-Schmidt operator, then $B\circ C$ and $C\circ B$ are of trace class, and
\[\mathrm{trace}(B\circ C)=\mathrm{trace}(C\circ B).\]
\end{proposition}

Assume $\Omega$ is a bounded domain in $\mathbb{R}^n$ with smooth boundary.

\begin{proposition}\label{prop_trace}
If $B$ is an operator of trace class on $L^2(\Omega)$, given by
\[(Bf)(x)=\int_{\Omega}K(x,y)f(y)\mathrm{d}y,\]
for $f\in L^2(\Omega)$.
If $K(\cdot\,,\,\cdot)$ is continous in $\Omega\times\Omega$, then
\[\mathrm{trace}(B)=\int_{\Omega}K(x,x)\mathrm{d}x.\]
\end{proposition}
\begin{proposition}
If $B$ and $C$ are two Hilbert-Schmidt operators, with kernels $K_1$ and $K_2$, then the kernel of $D=B\circ C$ is given by
\[K(x,y)=\int_{\Omega}K_1(x,z)K_2(z,y)\mathrm{d}z.\]
\end{proposition}

Define the trace of $T^s,\,s\in\mathbb{Z}^+$ as
\[
\tau_s(\rho)=\text{trace}(T^s).
\]
Clearly $\tau_s(\rho)=\sum_{k=1}^\infty \lambda_k^{-s}$. 
By above propositions, we have
\begin{equation}\label{trace}
\tau_s(\rho)=\int_0^1\cdots\int_0^1\rho(x_1)g(x_1,x_2)\cdots \rho(x_s)g(x_s,x_1)\mathrm{d}x_1\cdots\mathrm{d}x_s.
\end{equation}
We design an algorithm for the inverse spectral problem $(\ref{eq1})$ based on inverting the map
\[
\rho\rightarrow\{\tau_s(\rho)\}_{s=1}^\infty,
\]
which gives a more explicit relation between the function $\rho$ and the data than the map $(\ref{map})$.
The following proposition establishes the equivalence of the data $\{\tau_s(\rho)\}_{s=1}^\infty$ and $\{\lambda_k(\rho)\}_{k=1}^\infty$. To be general, we allow some eigenvalues to be non-simple.
\begin{proposition}
There is a one-to-one correspondence between $\{\lambda_k\}_{k=1}^\infty$ and $\{\tau_s\}_{s=1}^\infty$.
\end{proposition}
\begin{proof}
Assume the eigenvalues are ordered, i.e.,
\[
\lambda_1\leq\lambda_2\leq\lambda_3\leq\cdots.
\]
 Suppose the first eigenvalue $\lambda_1=\lambda_2=\cdots=\lambda_{m_1}$ has multiplicity $m_1$, $\lambda_{m_1}<\lambda_{m_1+1}$, then
\[
\lim_{s\rightarrow\infty}\sqrt[\leftroot{1}\uproot{20}s]{\sum_{k=1}^\infty\lambda_k^{-s}}=\lim_{s\rightarrow\infty}\sqrt[\leftroot{-1}\uproot{20}s]{m_1\lambda_1^{-s}+\sum_{k=m_1+1}^\infty\lambda_k^{-s}}=\lambda_1^{-1}\lim_{s\rightarrow\infty}\sqrt[\leftroot{-1}\uproot{20}s]{m_1+\sum_{k=m_1+1}^\infty\left(\frac{\lambda_1}{\lambda_k}\right)^s}=\lambda_1^{-1}
\]
and
\[
\lim_{s\rightarrow\infty}\lambda_1^s\times\sum_{k=1}^\infty\lambda_k^{-s}=m_1.
\]
We have already determined $\lambda_1$ and its multiplicity $m_1$. Then we exclude them from the traces and iterate.
\end{proof}

\begin{remark}
One can easily check the above proof and see that actually the map $\{\lambda_n\}_{n=1}^\infty\rightarrow\{\tau_s\}_{s=S}^\infty$ is injective for any $S\geq 1$. For some (positive self-adjoint) second-order elliptic differential operator $\mathcal{M}(\rho)$ in higher dimensions, the Green's function $G(x,y)$ might not be continuous on the diagonal $x=y$, and therefore we can not apply Proposition $\ref{prop_trace}$ to $\mathcal{M}(\rho)^{-1}$. However, we would not lose any information if we just use $\mathrm{trace}(\mathcal{M}(\rho)^{-s})$, $s=2,3,\cdots$.
\end{remark}

\section{Algorithm}\label{algorithms}
\subsection{Inversion scheme}
We are not to directly invert the formula $(\ref{trace})$, where multiple integral needs to be done. Instead, we will take the advantage of semi-separability of Green's function $g(x,y)$ for $-\Delta_D$ to simplify the calculation. Assume $\{\mu_n,\phi_n(x)\}_{n=1}^\infty$ are the eigenvalues and eigenfunctions of $-\Delta_D$, where
\[
-\Delta_D\phi_n=-\phi_n''=\mu_n\phi_n,
\]
\[
\mu_n=n^2\pi^2,\quad\phi_n(x)=\sqrt{2}\sin n\pi x.
\]
\begin{proposition}[Mercer's Theorem]
Assume an operator $B$ is positive definite, and its associated kernel $K(x,y)$ is a real-valued symmetric, continuous function of $x$ and $y$. Then $K$ can be expanded in a uniformly convergent series
\[
K(x,y)=\sum_{n=1}^\infty\kappa_n\phi_n(x)\phi_n(y),
\]
where $\kappa_n$ and $\phi_n$ are the eigenvalues and normalized eigenfunctions of $B$.
\end{proposition}
Applying Mercer's theorem to $A=(-\Delta_D)^{-1}$ whose kernel is $g(x,y)$, we have
\begin{equation}\label{gxy}
g(x,y)=\sum_{n=1}^\infty\mu_n^{-1}\phi_n(x)\phi_n(y).
\end{equation}
It is easy to see from $(\ref{trace})$ and $(\ref{gxy})$ that
\[
\begin{split}
\tau_s(\rho)=\sum_{n_1,\cdots,n_s}\frac{1}{\pi^{2s}n_1^2\cdots n_s^2}\left(\int_0^1\phi_{n_1}(x)\phi_{n_2}(x)\rho(x)\mathrm{d}x\right)\times\cdots\times\left(\int_0^1\phi_{n_s}(x)\phi_{n_1}(x)\rho(x)\mathrm{d}x\right).
\end{split}
\]
Denote $\mathbf{M}(\rho)$ to be an infinite-dimensional matrix where
\[
\mathbf{M}_{ij}(\rho)=\frac{1}{\pi^2ij}\int_0^1\phi_{i}(x)\phi_{j}(x)\rho(x)\mathrm{d}x.
\]
Clearly, we have
\begin{equation}\label{trace1}
\sum_{k=1}^\infty \lambda_k^{-s}=\tau_s(\rho)=\text{trace}(\mathbf{M}^s(\rho)).
\end{equation}
The Fr\'echet derivative $\tau_s'$ of $\tau_s(\rho)$ at $\rho^0$ is
\[
\tau_s'[\rho^0](\delta \rho)=s\times\mathrm{trace}(\mathbf{M}^{s-1}(\rho^0)\mathbf{M}(\delta\rho)).
\]

If we use the approximation of $\rho$ under certain basis $\{\psi_m\}_{m=1}^\infty$:
\begin{equation}\label{cosine_expansion}
\rho(x)=\sum_{m=1}^M a_m\psi_m(x),
\end{equation}
and denote $\mathbf{a}=\{a_1,a_2,\cdots,a_M\}$. With a little abuse of notations, we use $\tau_s(\mathbf{a}),\,\mathbf{M}(\mathbf{a})$ in place of  $\tau_s(\rho),\,\mathbf{M}(\rho)$ in the following. Then
\[
\mathbf{M}_{ij}(\mathbf{a})=\sum_{m=1}^M 2a_m \mathbf{M}_{ij}(\mathbf{e}_m),
\]
where $\mathbf{M}_{ij}(\mathbf{e}_m)=\int_0^1\sin i\pi x\,\sin j\pi x\,\psi_m(x)\mathrm{d}x$.

 If we use truncated Fourier cosine series,
 \begin{equation}\label{cosine_expansion1}
\rho(x)=\sum_{m=1}^M a_m\cos 2(m-1)\pi x,
\end{equation}
  then
\[
\begin{split}
\mathbf{M}_{ij}(\mathbf{e}_m)=&\frac{1}{\pi^2ij}\int_0^1\sin i\pi x\sin j\pi x\cos2(m-1)\pi x\mathrm{d}x\\
=&\begin{cases}
-\frac{1}{4 \pi^2ij},\quad & i+j+2m-2=0,\\
-\frac{1}{4 \pi^2ij},\quad & i+j-2m+2=0,\\
\frac{1}{4 \pi^2ij},\quad & i-j+2m-2=0,\,m\neq 1,\\
\frac{1}{4 \pi^2ij},\quad & i-j-2m+2=0,\,m\neq 1,\\
\frac{1}{2 \pi^2ij},\quad & i=j,m=1,\\
0,\quad &\text{otherwise},
\end{cases}
\end{split}
\]
and $\mathbf{e}_m=\{a_1=0,\cdots,a_{m-1}=0,a_m=1,a_{m+1}=0,\cdots,a_M=0\}$.

Now we have
\[
\frac{\partial\tau_s}{\partial a_m}[\mathbf{a}^0]=s\times\mathrm{trace}(\mathbf{M}^{s-1}(\mathbf{a}^0)\mathbf{M}(\mathbf{e}_m)).
\]
For numerical implementation, we also need truncate $\mathbf{M}$ to be a $J\times J$ matrix.

In general, the Jacobian for the map
\[
\rho\rightarrow\{\tau_s(\rho)\}_{s=1}^S=\{\mathrm{trace}((A\circ M_\rho)^s)\}_{s=1}^S
\]
is super ill-posed. The fundamental reason is that it is {\em inappropriate} to measure the residual as
\[
\sum_{s=1}^S(\tau_s(\rho^{\mathrm{computed}})-\tau_s(\rho^{\mathrm{true}}))^2.
\]
 For a numerically more stable scheme, we will instead invert a different map
\begin{equation}\label{themap}
\rho\rightarrow\{\mathrm{trace}(\mathcal{P}_n(A\circ M_\rho))\}_{n=1}^N,
\end{equation}
for properly chosen polynomials $\{\mathcal{P}_n\}_{n=1}^N$, with $\mathcal{P}_n(0)=0$. The following proposition, immediately from $(\ref{trace1})$, is the corner stone for the algorithm:
\begin{proposition}
If $\mathcal{P}_n$ is a polynomial with $\mathcal{P}_n(0)=0$. Then
\begin{equation}
\sum_{k=1}^\infty \mathcal{P}_n(\lambda_k^{-1})=\mathrm{trace}(\mathcal{P}_n(A\circ M_\rho))=\mathrm{trace}(\mathcal{P}_n(\mathbf{M}(\mathbf{a}))).
\end{equation}
\end{proposition}
Here $\mathcal{P}_n(\mathbf{M}(\mathbf{a}))$ is some infinite-dimensional matrix, and its trace is just the sum of its diagonal entries.
One can see that computing $\mathcal{P}_n(\mathbf{M}(\mathbf{a}))$ just needs basic matrix operations.

\begin{remark}
We emphasize here that the choice of proper polynomials $\{\mathcal{P}_n\}_{n=1}^N$ is crucial for the algorithm. We want the associated Jacobian to be less ill-posed. The monomials $\{x,x^2,\cdots\}$ are actually very bad choices, and would lead to a failure of the algorithm.
\end{remark}

We will use the (shifted) Chebyshev polynomials on $[0,1]$ given by recursive formulas:
\[
T_{n+1}(x) = (4x-2)T_n(x)-T_{n-1}(x),\quad\quad n\geq 1,
\]
with
\[
T_0(x)=1, \quad\quad T_1(x)=2x-1.
\]
 Chebyshev polynomials have the property that
\[
\max_{x\in[0,1]}\vert T_n(x)\vert=1,
\]
which is favorable to our need.

Denote $\widetilde{T}_n(x)=xT_{n-1}(x)$, then we have the recursive formula
\[
\widetilde{T}_{n+1}(x)=(4x-2)\widetilde{T}_{n}(x)-\widetilde{T}_{n-1}(x)
\]
with
\[
\widetilde{T}_1(x)=x,\quad\quad \widetilde{T}_2(x)=2x^2-x.
\]
We will choose $\widetilde{T}_1,\widetilde{T}_2,\cdots, \widetilde{T}_N$ as the polynomials for our algorithm. Therefore, essentially we will be inverting the map
\[
\mathbf{a}\rightarrow\{\mathrm{trace}(\widetilde{T}_n(\mathbf{M}(\mathbf{a})))\}_{n=1}^N.
\]

For the use of Chebyshev polynomials, we first choose a positive number $t$ slightly smaller than $\lambda_1^{-1}$ to rescale $A\circ M_\rho$ to $t(A\circ M_\rho)$, such that $\lambda_{\max}(t(A\circ M_\rho))$ is slight smaller than $1$. The matrix $\mathbf{M(a)}$ in the following is the rescaled one. For convenience of understanding, one can just assume the first eigenvalue $\lambda_1$ is slightly larger than $1$.

For the evaluation of the forward map and its Jacobian, we will use the following recursive relations:
\[
\mathrm{trace}(\widetilde{T}_{n+1}(\mathbf{M}(\mathbf{a}))=\mathrm{trace}((4\mathbf{M}(\mathbf{a})-2)\widetilde{T}_{n}(\mathbf{M}(\mathbf{a})))-\mathrm{trace}(\widetilde{T}_{n-1}(\mathbf{M}(\mathbf{a}))),
\]
and
\[
\begin{split}
\frac{\partial\mathrm{trace}(\widetilde{T}_{n+1}(\mathbf{M}(\mathbf{a})))}{\partial a_m}&=\mathrm{trace}\left(\frac{\partial\widetilde{T}_{n+1}(\mathbf{M}(\mathbf{a}))}{\partial a_m}\right)\\
&=\mathrm{trace}\left(4\frac{\partial \mathbf{M}(\mathbf{a})}{\partial a_m}\widetilde{T}_{n}(\mathbf{M}(\mathbf{a}))+(4\mathbf{M}(\mathbf{a})-2)\frac{\partial \widetilde{T}_n(\mathbf{M}(\mathbf{a}))}{\partial a_m}\right)\\
&\quad\quad\quad-\mathrm{trace}\left(\frac{\partial \widetilde{T}_{n-1}(\mathbf{M}(\mathbf{a}))}{\partial a_m}\right).
\end{split}
\]
Notice
\[
\frac{\partial \mathbf{M}(\mathbf{a})}{\partial a_m}=2\mathbf{M}(\mathbf{e}_m).
\]
The scheme of the algorithm is then summarized in Algorithm \ref{algorithm}.\\

\begin{algorithm}
\caption{Inversion of trace formulas with Chebyshev polynomials}\label{algorithm}
\begin{algorithmic}[1]
\STATE precompute $\mathbf{M}(\mathbf{e}_m)$, traces $r_1^{true}=\sum_{k=1}^\infty\widetilde{T}_1(\lambda_k^{-1}),\ldots, r_N^{true}=\sum_{k=1}^\infty\widetilde{T}_N(\lambda_k^{-1})$
\STATE given initial guess $\mathbf{a}_0$
\FOR{$1\leq n\leq$ max number of iterations}
\STATE form $\widetilde{T}_1(\mathbf{M}(\mathbf{a}_{n-1}))$, $\widetilde{T}_2(\mathbf{M}(\mathbf{a}_{n-1}))$, $\frac{\partial \widetilde{T}_{1}(\mathbf{M}(\mathbf{a}_{n-1}))}{\partial a_m}$, $\frac{\partial \widetilde{T}_{2}(\mathbf{M}(\mathbf{a}_{n-1}))}{\partial a_m}$
\STATE $r_{1}=\textbf{trace}(\widetilde{T}_{1}(\mathbf{M}(\mathbf{a}_{n-1})))$
\STATE $r_{2}=\textbf{trace}(\widetilde{T}_{2}(\mathbf{M}(\mathbf{a}_{n-1})))$
\FOR{$1\leq m\leq M$}
\STATE $J_{1,m}=\textbf{trace}\left(\frac{\partial\widetilde{T}_{1}(\mathbf{M}(\mathbf{a}))}{\partial a_m}\right)$
\STATE $J_{2,m}=\textbf{trace}\left(\frac{\partial\widetilde{T}_{2}(\mathbf{M}(\mathbf{a}))}{\partial a_m}\right)$
\ENDFOR
\FOR{$2\leq j\leq N-1$}
\STATE $\widetilde{T}_{j+1}(\mathbf{M}(\mathbf{a}_{n-1}))=(4\mathbf{M}(\mathbf{a}_{n-1})-2)\widetilde{T}_{j}(\mathbf{M}(\mathbf{a}_{n-1}))-\widetilde{T}_{j-1}(\mathbf{M}(\mathbf{a}_{n-1}))$
\FOR{$1\leq m\leq M$}
\STATE $\frac{\partial\widetilde{T}_{j+1}(\mathbf{M}(\mathbf{a}))}{\partial a_m}=8\mathbf{M}(\mathbf{e}_m)\widetilde{T}_{j}(\mathbf{M}(\mathbf{a}_{n-1}))+(4\mathbf{M}(\mathbf{a}_{n-1})-2)\frac{\partial \widetilde{T}_j(\mathbf{M}(\mathbf{a}_{n-1}))}{\partial a_m}-\frac{\partial \widetilde{T}_{j-1}(\mathbf{M}(\mathbf{a}_{n-1}))}{\partial a_m}$
\ENDFOR
\STATE $r_{j+1}=\textbf{trace}(\widetilde{T}_{j+1}(\mathbf{M}(\mathbf{a}_{n-1})))$
\FOR{$1\leq m\leq M$}
\STATE $J_{j+1,m}=\textbf{trace}\left(\frac{\partial\widetilde{T}_{j+1}(\mathbf{M}(\mathbf{a}_{n-1}))}{\partial a_m}\right)$
\ENDFOR
 \ENDFOR
 \STATE compute $\delta \mathbf{a}$ using Jacobi $\mathbf{J}=(J_{j,m})_{N\times M}$ and residual $\mathbf{r}^{true}-\mathbf{r}=(r_j^{true}-r_j)_{N\times 1}$
 \STATE $\mathbf{a}_n=\mathbf{a}_{n-1}+\delta\mathbf{a}$
 \ENDFOR
\end{algorithmic}
\end{algorithm}
\subsection{Details on the algorithm}
~\\

\noindent \textbf{Numerical trace computation.} For the computation of $\sum_{k=1}^\infty \widetilde{T}_n(\lambda_k^{-1})$ for each $n$, we also need to first use the similar recursive scheme to compute $\widetilde{T}_n(\lambda_k^{-1})$, and then do the summation over $m$. One shall {\bf never} first calculate $\tau_s=\sum_{k=1}^\infty \lambda_k^{-s},\, s=1,\cdots, n$ and then compute $a_n\tau_n+a_{n-1}\tau_{n-1}+\cdots+a_1\tau_1$ using the explicit form of $\widetilde{T}_n(x)=a_nx^n+a_{n-1}x^{n-1}+\cdots+a_1x$, although those coefficients $\{a_\beta\}_{\beta=1}^n$ can be precomputed. The main reasons are: 1) the coefficients $\{a_\beta\}_{\beta=1}^n$ for high order Chebyshev polynomials vary drasically in magnitude, so one will easily encounter overflow/underflow problems; 2) the sum $\sum_{k=1}^\infty\lambda^{-s}_k$ for large $s$ will be dominated by the sum of the first few eigenvalues, rounding the ones following. This is a typical issue for numerically adding a large number and a small one.\\

\noindent \textbf{Use of Chebyshev polynomials. } If we use $K$ eigenvalues as the data, basically we need the highest degree of used Chebyshev polynomials $N=\mathcal{O}(K^2)$. The reason is that we need to make sure the chosen Chebyshev polynomials can discriminate all those eigenvalues, and we have the estimates $\frac{\lambda_K}{\lambda_1}\sim\mathcal{O}(K^2)$. We remark here that the use of \textbf{ALL} Chebyshev polynomials with degree lower than $N$ is totally superfluous. However, we do not know what are the ``optimal" choices of Chebyshev polynomials. This is subject to further investigation, and we simply use all the polynomials for the numerical experiments presented later.\\


\noindent \textbf{Traces using finite spectral data. } Note that traces involve infinite sums
\[
\mathrm{trace}(\mathcal{P}_n(A\circ M_\rho))=\sum_{k=1}^K\mathcal{P}_n(\lambda_k^{-1})+\sum_{k=K+1}^\infty\mathcal{P}_n(\lambda_k^{-1}).
\]
If only $K$ eigenvalues are known, we can approximate $\lambda_k$ for $k\geq K+1$ using the asymptotics \cite{rundell1992reconstruction1}
\[
\lim_{k\rightarrow+\infty}\frac{\lambda_k}{k^2\pi^2}=L^{-2}.
\]
We first get an approximation $\widetilde{L}$ of $L$ from the first $K$ eigenvalues, and then use
\[
\sum_{k=K+1}^{K_1}\mathcal{P}_n\left(\frac{\widetilde{L}^2}{k^2\pi^{2}}\right)
\]
to approximate $\sum_{k=K+1}^{\infty}\mathcal{P}_n(\lambda_k^{-1})$ with $K_1>K$. So basically, we are using $K$ ``true" eigenvalues and $K_1-K$ ``approximated eigenvalues".\\

\noindent \textbf{Multistep optimization.}
It is natural to start with a small number of basis functions, solve the optimization problem, and use the result as the initial guess for the next iteration with more basis functions. It is a good strategy to avoid local minimums. And thus we also do not need to worry too much about that we did not include the important constraint:
\[
\rho(x)>0,
\]
for optimization.

Figure \ref{multistep1} shows a two-step reconstruction for the function $\rho(x)=1-0.3e^{-20(x-0.5)^2}$. We first use only $K=3$ eigenvalues and $M=3$ basis functions to do an approximation. Then we use the approximation as the initial guess for the second-step reconstruction with $K=7$ eigenvalues and $M=7$ basis functions.\\
\begin{figure}[ht]
\includegraphics[height=6cm]{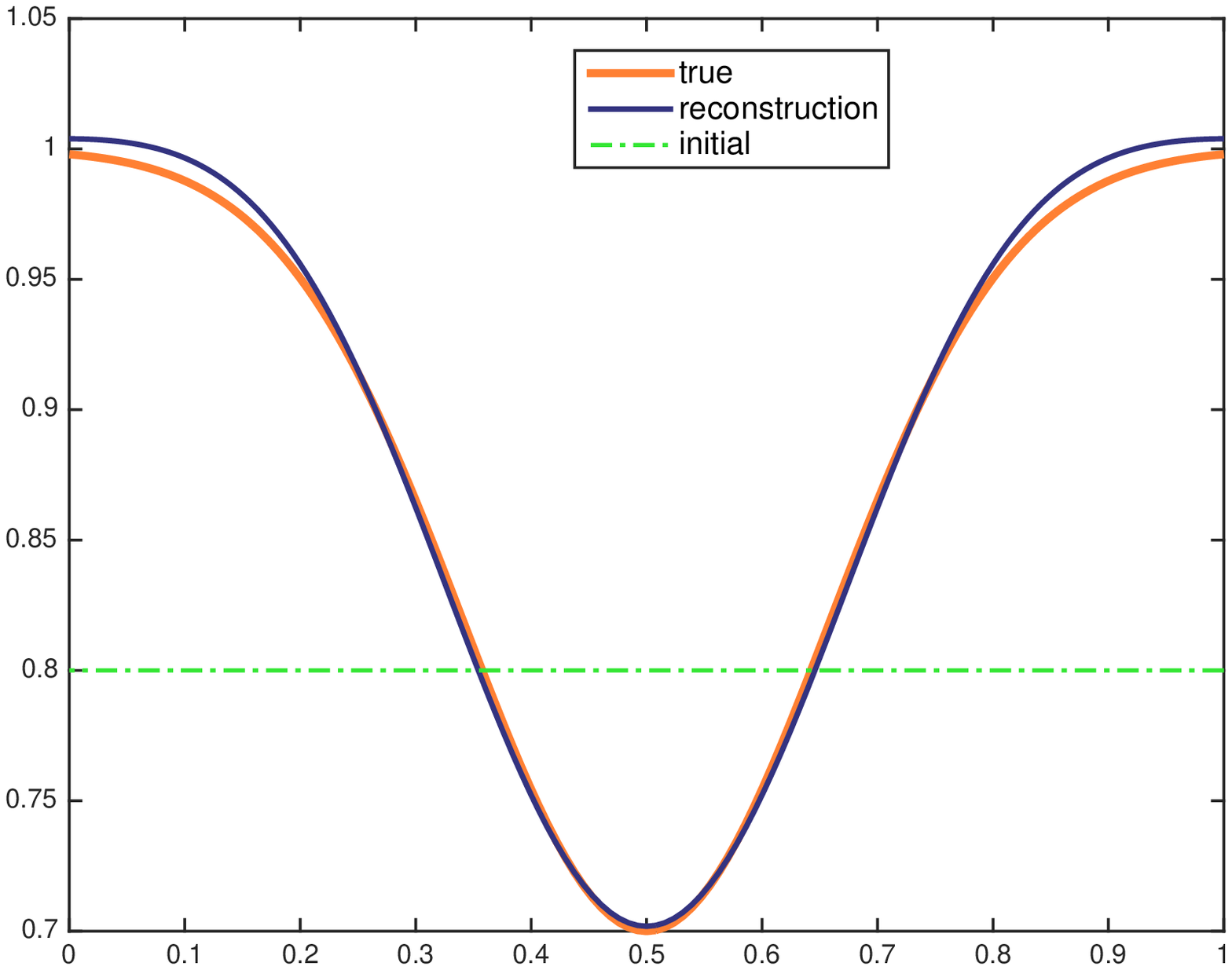}~~
\includegraphics[height=6cm]{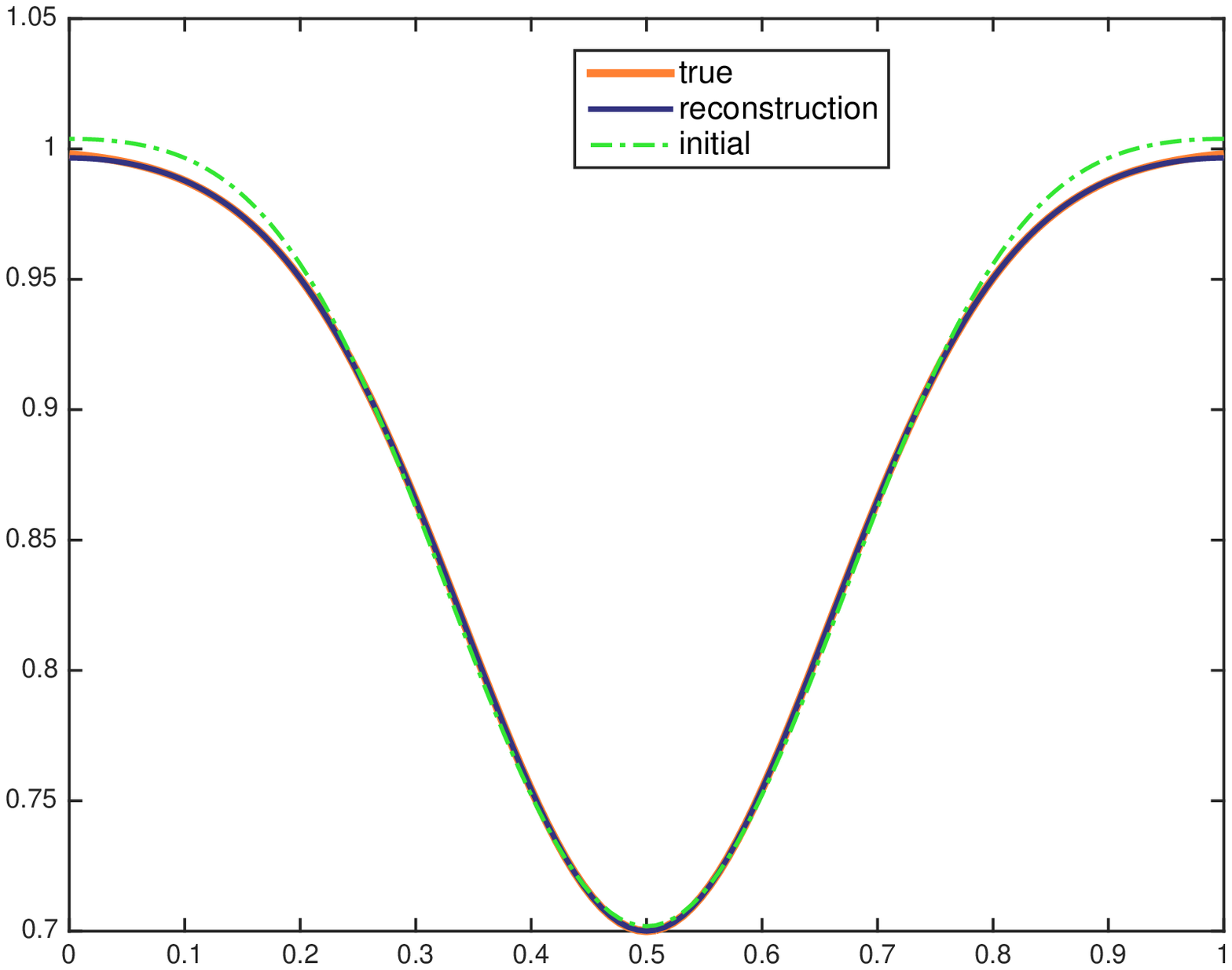}
\caption{Left: $M=K=3, \,J=3,\, N=20$; Right: $M=K=7,\, J=15, \,N=300$. Notice that the reconstructed profile on the left is the initial guess for the reconstruction on the right.}\label{multistep1}
\end{figure}

\noindent \textbf{Condition number.}
We generate Jacobian matrices with $N=1000$ Chebyshev polynomials at $\rho(x)\equiv 1$. The condition numbers, for $M=1,\,2,\cdots,\, 10$, are plotted in Figure \ref{cond_number}, compared with the values of $M^2$, $M^3$ and $M^4$. We can see that the condition number is approximately of order $\mathcal{O}(M^3)$. This is only slightly worse than the condition number $\mathcal{O}(M^2)$ for the method introduced in \cite{lowe1992recovery}. But again, we remark here that probably one can choose other polynomials to result in a better condition number.
\begin{figure}[ht]
\includegraphics[height=8cm]{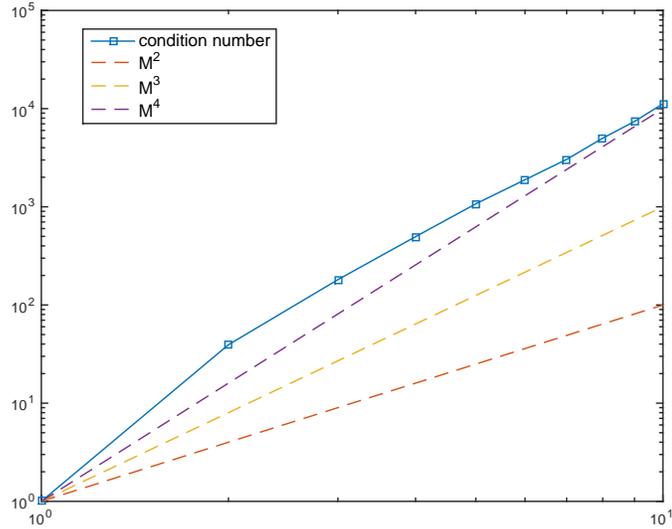}
\caption{Condition number of the Jacobian at $\rho\equiv 1$}\label{cond_number}
\end{figure}
\section{Numerical experiment}\label{numerical}

In this section we generate synthetic data to conduct some numerical experiments. To get ``exact" eigenvalues, we use Chebyshev pseudospectral collocation discretization of the Laplacian operator $\Delta=\frac{\mathrm{d}^2}{\mathrm{d}x^2}$, using Trefethen's $\mathtt{cheb.m}$ routine \cite{trefethen2000spectral}:
\begin{lstlisting}
  %differentiation matrix for [-1,1]
  [D,x] = cheb(N);  
  %convert domain to [0,1]
  x = (x+1)/2; 
  x = x';
  % A is the Laplace operator
  A = D^2; 
  %impose Dirichlet boundary condition
  A = A(2:end-1,2:end-1); 
\end{lstlisting}
We use $200$ Chebyshev points to discretize Laplacian. At least $40$ percent of generated eigenvalues are ``exact". The number of reliable eigenvalues that can be computed via different numerical discretizations is discussed in \cite{zhang2015many}.

For all computations, Gauss-Newton is used as the optimization algorithm with tolerance set to $10^{-5}\times N$. For every example, multistep optimization is taken. The initial guess for the first step is close to the mean of the true density $\rho(x)$.

The parameters in the algorithm are listed in Table \ref{Symbol} with the typical choices of their values. Some choices have been explained above, some will be explained below.

\begin{table}[h]
  \begin{tabular}{|c | c | c |}
  \hline
 notation & parameter& typical choice of value\\
  \hline
$K$ & number of ``true" eigenvalues used & --------- \\
    \hline
 $M$&number of basis functions& $M=K$\\
  \hline
  $J$ & $J\times J$: size of truncated matrix $\mathbf{M}$ & $J>K$\\
  \hline
    $N$ & highest degree of Chebyshev polynomials  & $N=\mathcal{O}(K^2)$\\
  \hline
  $K_1$ &$K_1-K$: number of ``approximate" eigenvalues used& $K_1=J$\\
  \hline
  \end{tabular}
  \caption{Parameters for the algorithm}\label{Symbol}
\end{table}

 \subsection{Example 1}
 For the first experiment, we recover the density function
 \[
\rho_1(x)=1+0.3\cos 2\pi x+0.2\cos 4\pi x+0.15\cos 6\pi x-0.1\cos 8\pi x-0.05\cos 10\pi x+0.02\cos 12\pi x,
\]
which is a linear combination of the first $M=7$ Fourier cosine functions. We will use exactly $7$ Fourier cosine functions to do the reconstruction. We take the number of eigenvalues used $K$ to be equal to $M=7$ to maximize the amount of information utilized. We will use $N=200$ Chebyshev polynomials, and the matrix $\mathbf{M}$ is truncated to be of size $J\times J$. We study how different values of $J$ affect the behavior of the algorithm. For the approximation of traces using only $7$ eigenvalues, we take $K_1 =J$. The numerical results are shown in Figure \ref{E1}.

We list the first $K=7$ ``exact" eigenvalues and the eigenvalues for the reconstructed densities in Table \ref{T1}. One can see more eigenvalues for the reconstructed density are close enough to the ``exact" ones with increasing $J$. Therefore we need to take $J>M=K$ to do reconstruction.
\begin{figure}
    \begin{subfigure}[t]{0.42\textwidth}
      \includegraphics[width=\textwidth]{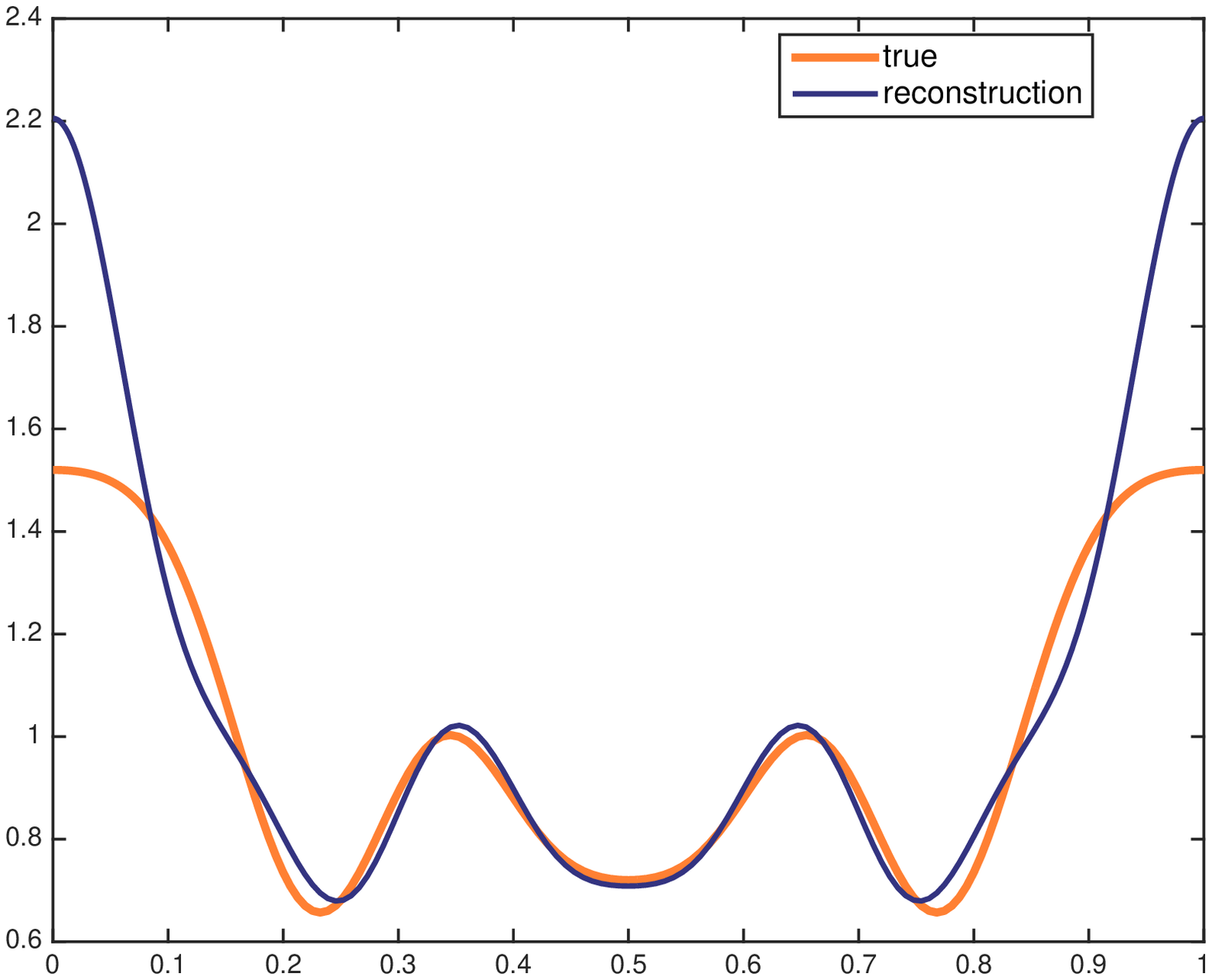}
      \caption{recovery with $J=7$}
    \end{subfigure}
    \begin{subfigure}[t]{0.42\textwidth}
      \includegraphics[width=\textwidth]{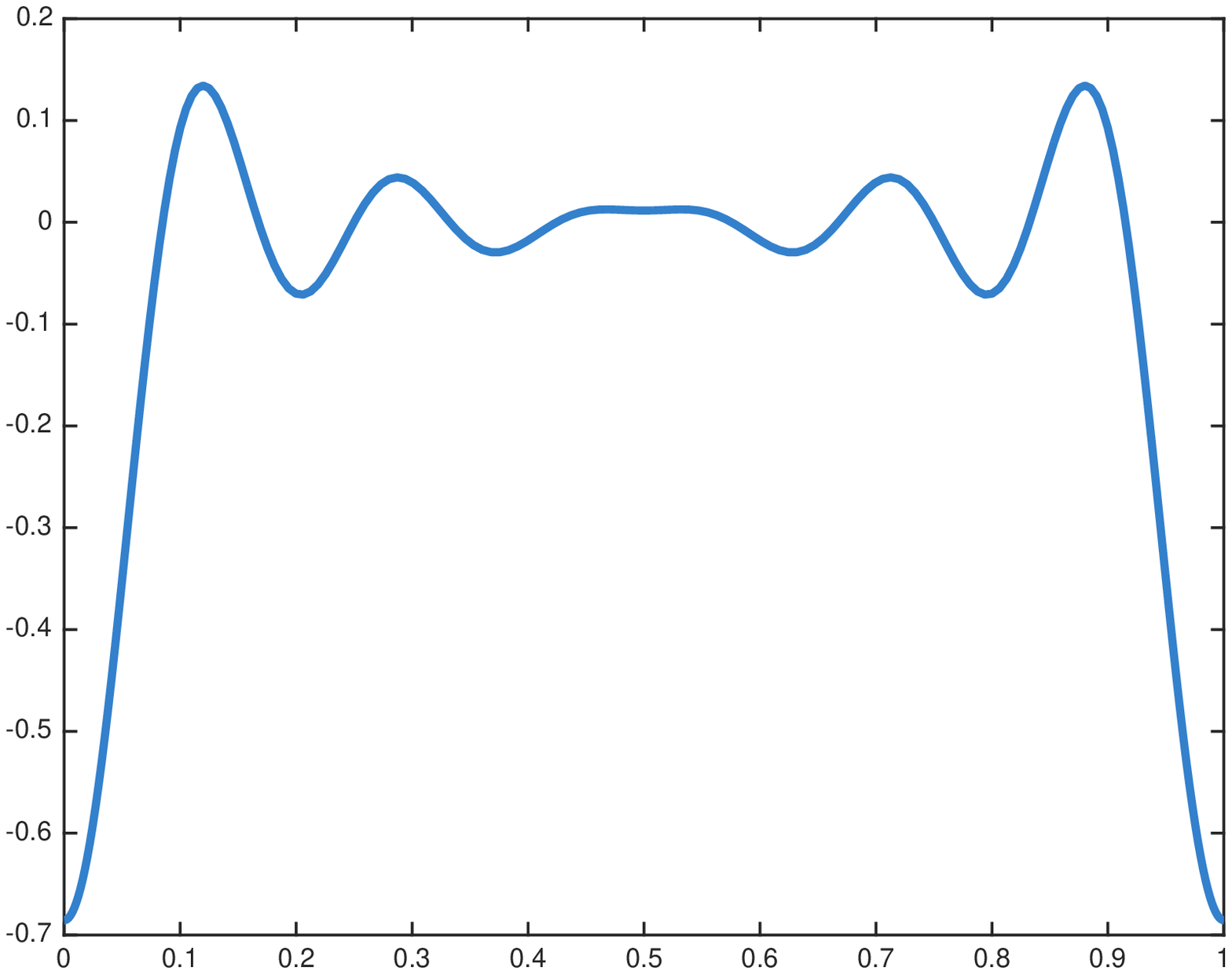}
      \caption{mismatch with $J=7$}
    \end{subfigure}
     \begin{subfigure}[t]{0.42\textwidth}
      \includegraphics[width=\textwidth]{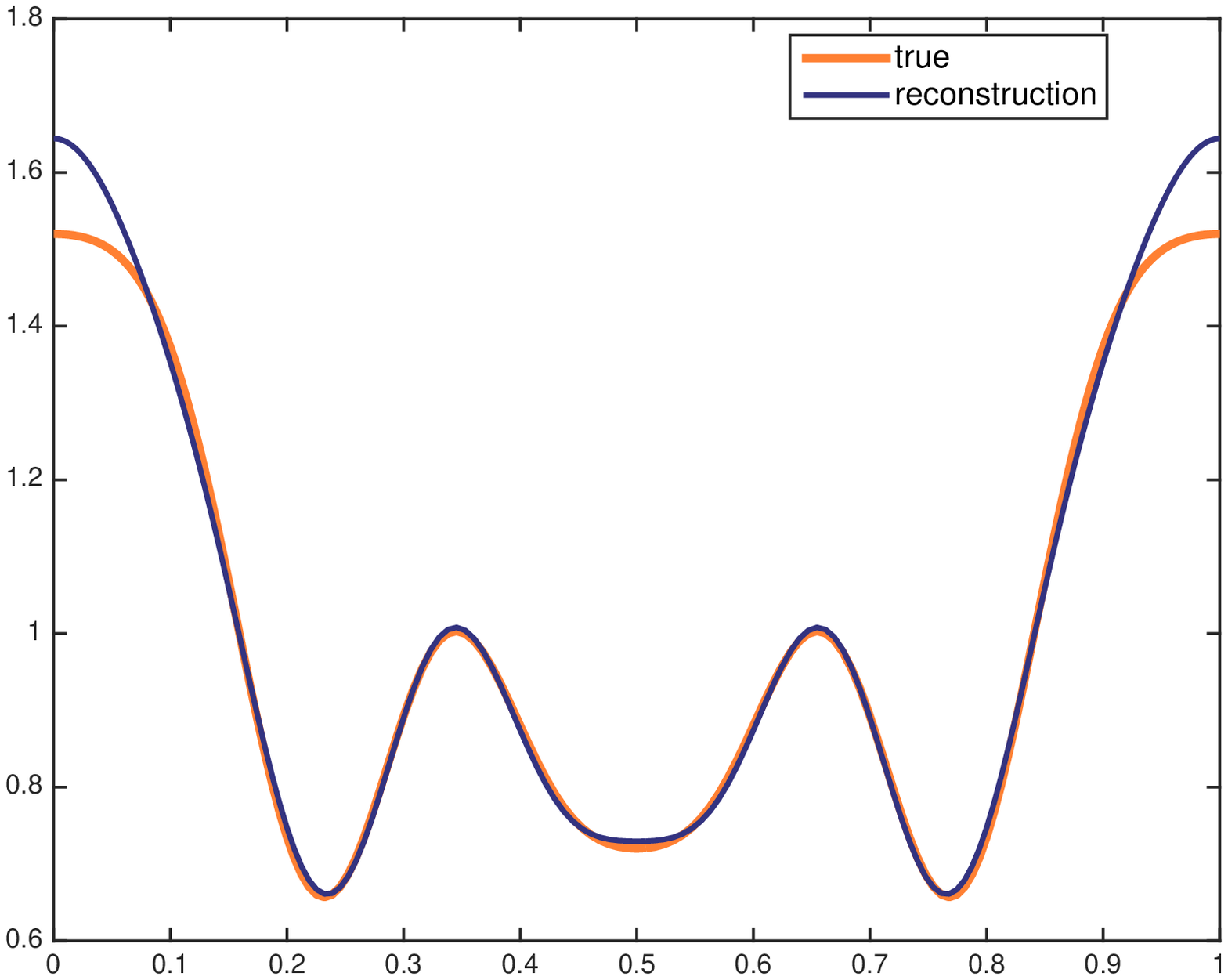}
      \caption{recovery with $J=10$}
    \end{subfigure}
    \begin{subfigure}[t]{0.42\textwidth}
      \includegraphics[width=\textwidth]{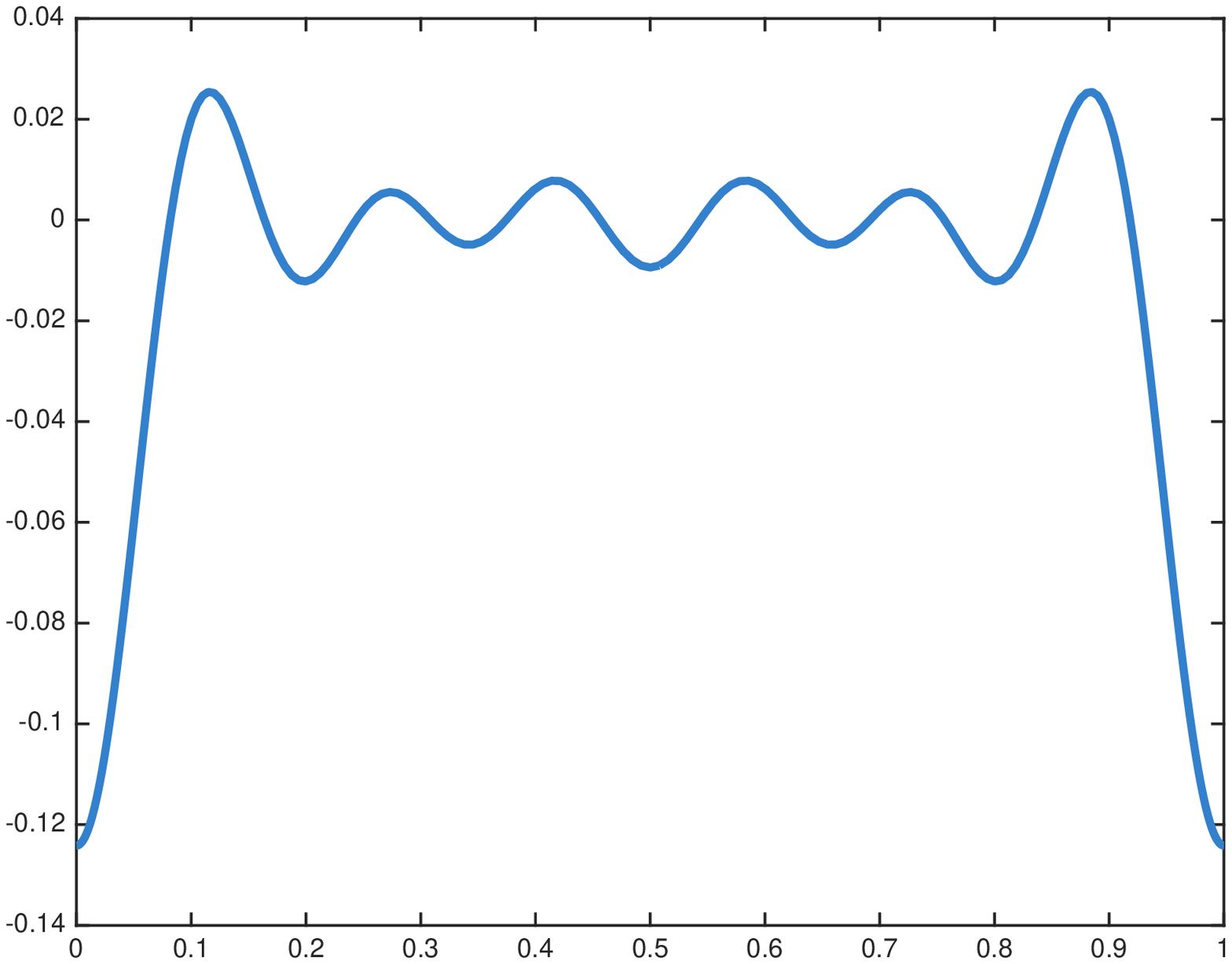}
      \caption{mismatch with $J=10$}
    \end{subfigure}
         \begin{subfigure}[t]{0.42\textwidth}
      \includegraphics[width=\textwidth]{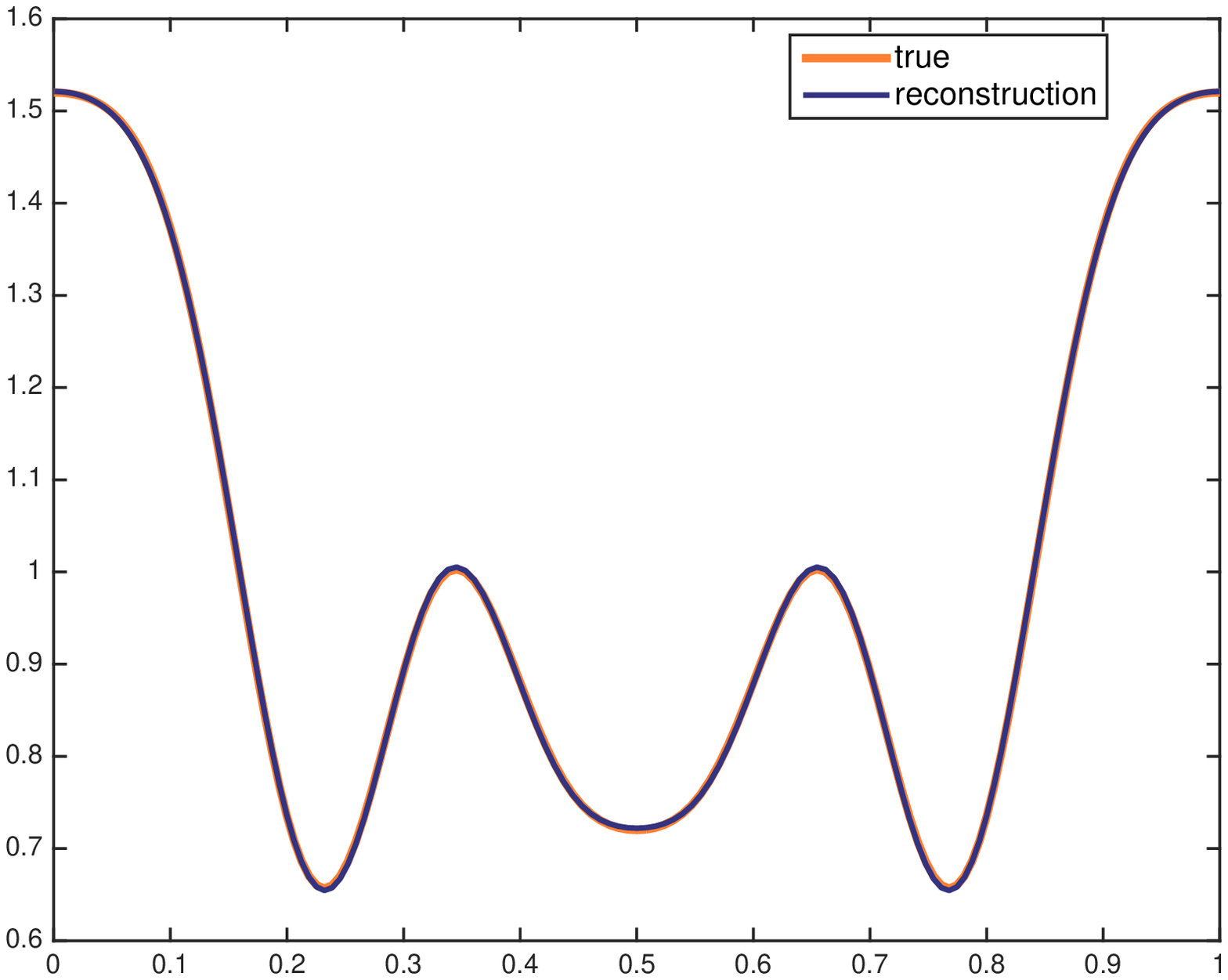}
      \caption{recovery with $J=15$}
    \end{subfigure}
    \begin{subfigure}[t]{0.42\textwidth}
      \includegraphics[width=\textwidth]{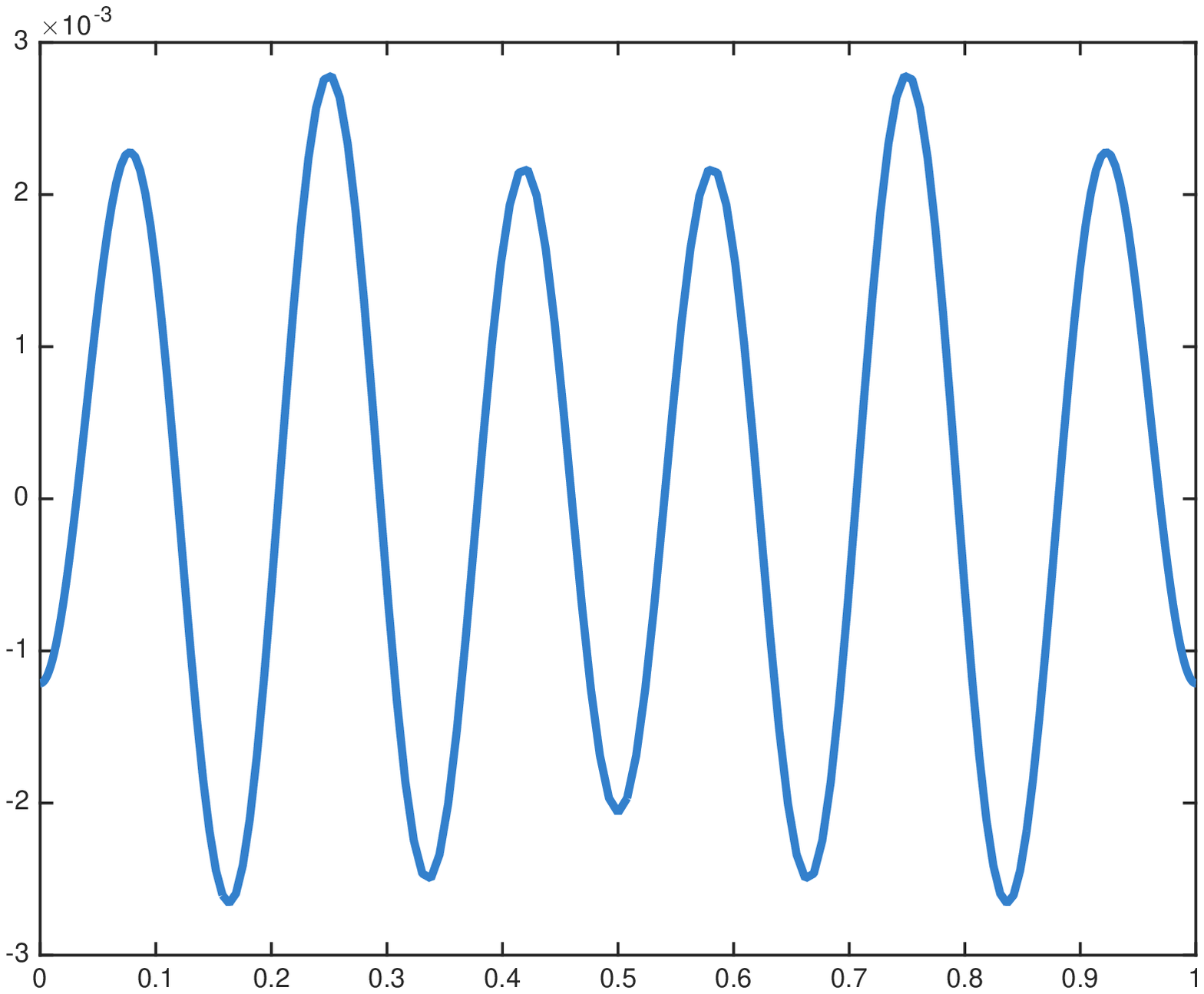}
      \caption{mismatch with $J=15$}
    \end{subfigure}
    \caption{Recovery of $\rho_1(x)$ with $J=7,10,15$.}\label{E1}
  \end{figure}

\begin{table}[h]
  \begin{tabular}{|c | c | c | c | c | c | c | c |}
  \hline
 &  &  \multicolumn{2}{c|}{$J=7$}&  \multicolumn{2}{c|}{$J=10$}& \multicolumn{2}{c|}{$J=15$}\\
    \hline
   & true &  eigenvalue & mismatch&  eigenvalue & mismatch&  eigenvalue & mismatch \\ \hline
$\lambda_1$  &11.6001 &11.5999   & 0.0003  & 11.6000 &   0.0002   &11.6001   & 0.0000\\
  $\lambda_2$&43.5488& 43.5010  &  0.0478  & 43.5483   & 0.0004  & 43.5486    &0.0002\\
 $\lambda_3$&93.3770&  93.1545   & 0.2225  & 93.3736  &  0.0034  & 93.3736  &  0.0034\\
$\lambda_4$&148.5702&  147.1728  &  1.3974&  148.5736  & -0.0034 & 148.5586   & 0.0117\\
 $\lambda_5$&253.724& 249.2582  &  4.4657 & 252.6704  &  1.0535  &253.8159  & -0.0920\\
  $\lambda_6$&373.8342&351.9368  & 21.8974&  371.0629  &  2.7713 & 374.2342  & -0.4000\\
 $\lambda_7$& 493.2769&463.8748  & 29.4020&  487.0575  &  6.2194  &493.0034  &  0.2735\\
  \hline
  \end{tabular}
  \caption{True eigenvalues and the eigenvalues for reconstructed densities.}\label{T1}
\end{table}

\subsection{Example 2}
For the second experiment, we recover
\[
\rho_2(x)=1+(x-0.5)^2,
\]
using $K=5,\,10,\,15$ eigenvalues. We will use $M=K$ Fourier cosine functions to do the each reconstruction. We use $N=1000$ Chebyshev polynomials, and the matrix $\mathbf{M}$ is truncated to be of size $20\times 20$ ($J=20$). The numerical results are shown in Figure \ref{E2}. 

We also compare the reconstructed density and the Fourier approximation (with the same number of basis functions) of the true density $\rho_2(x)$. Results are shown in Figure \ref{E3}. We see that the misfit of the reconstruction with the Fourier approximation of the density is smaller than with the exact density. 

\begin{figure}
    \begin{subfigure}[t]{0.42\textwidth}
      \includegraphics[width=\textwidth]{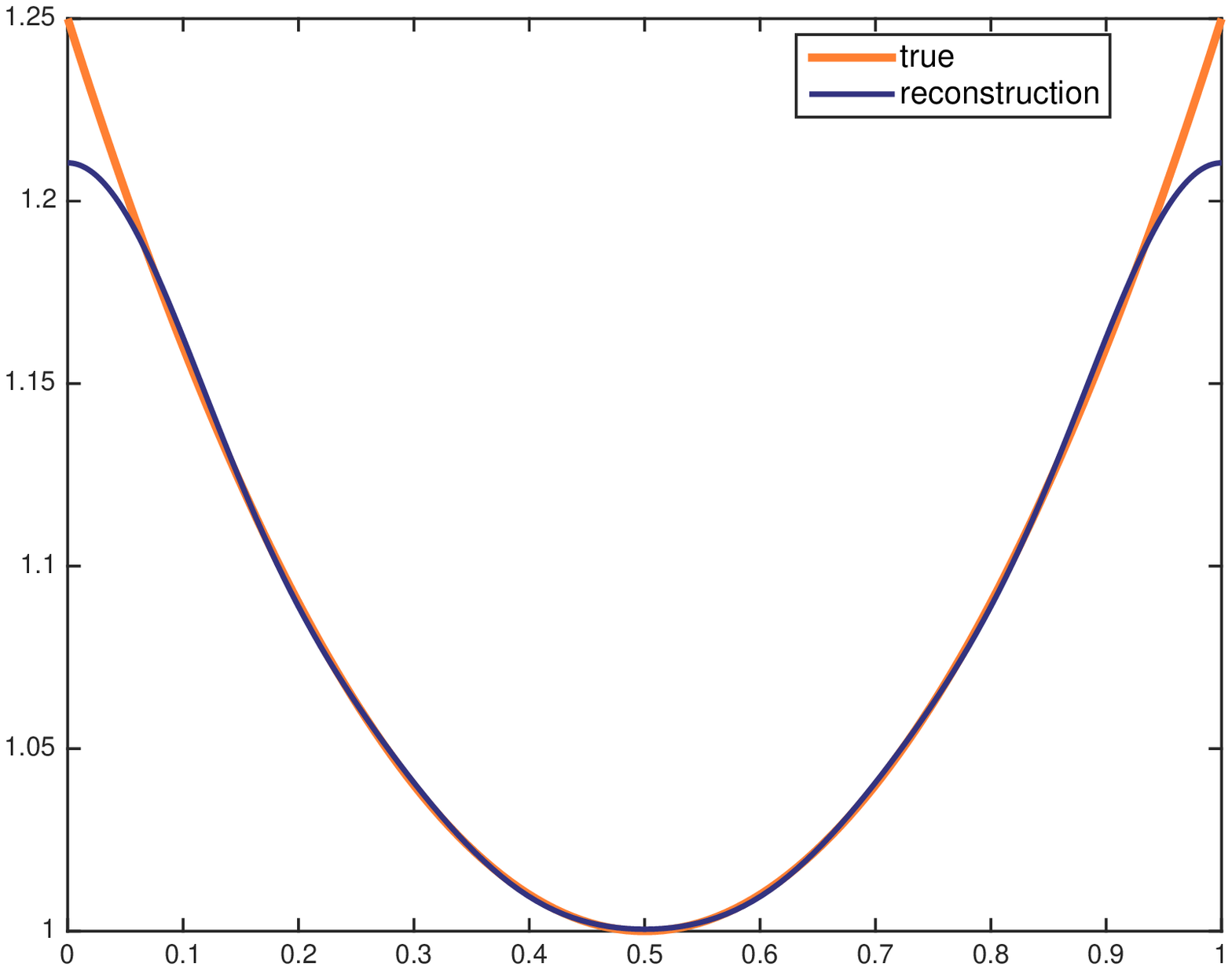}
      \caption{recovery using $5$ eigenvalues}
    \end{subfigure}
    \begin{subfigure}[t]{0.42\textwidth}
      \includegraphics[width=\textwidth]{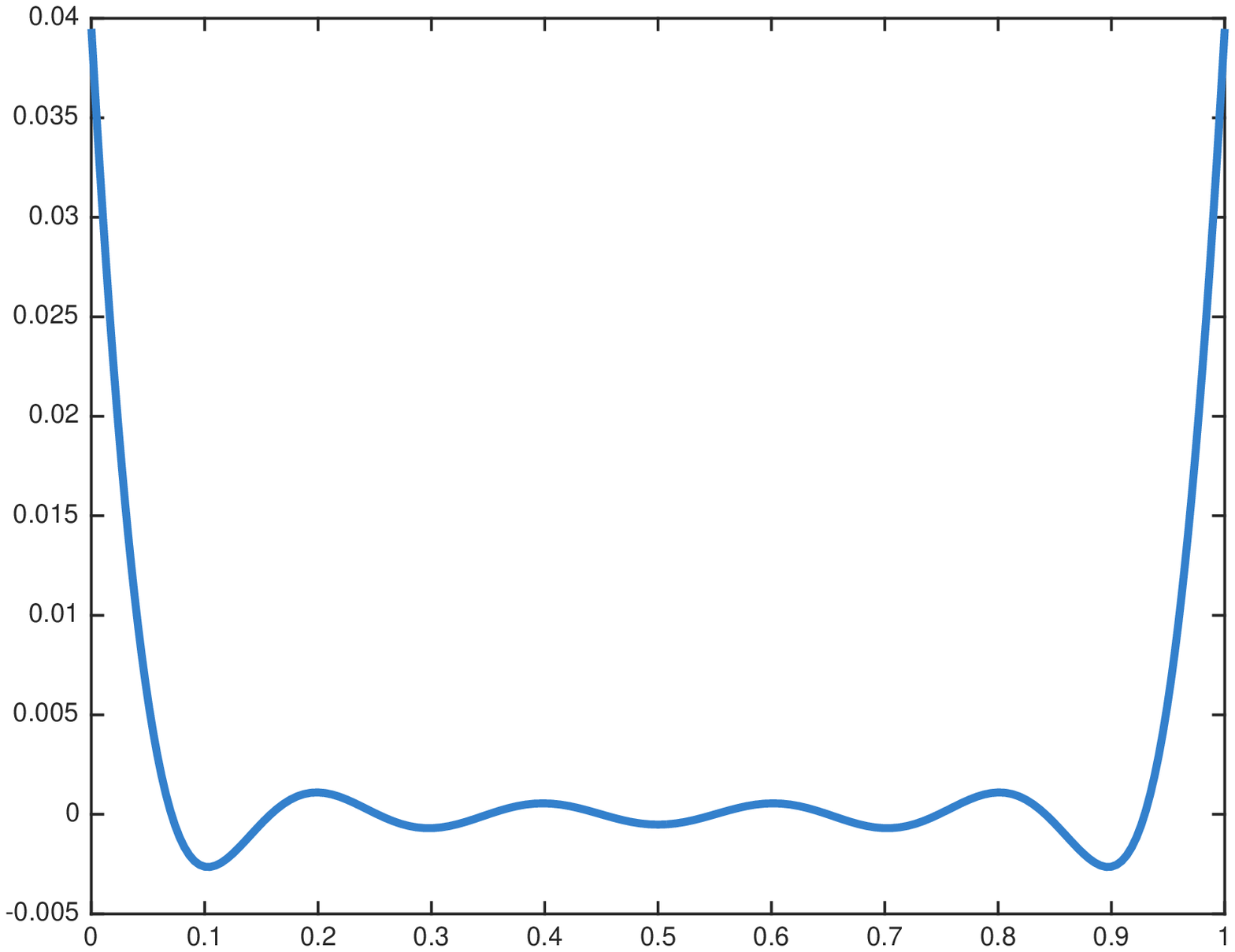}
      \caption{mismatch using $5$ eigenvalues}
    \end{subfigure}
         \begin{subfigure}[t]{0.42\textwidth}
      \includegraphics[width=\textwidth]{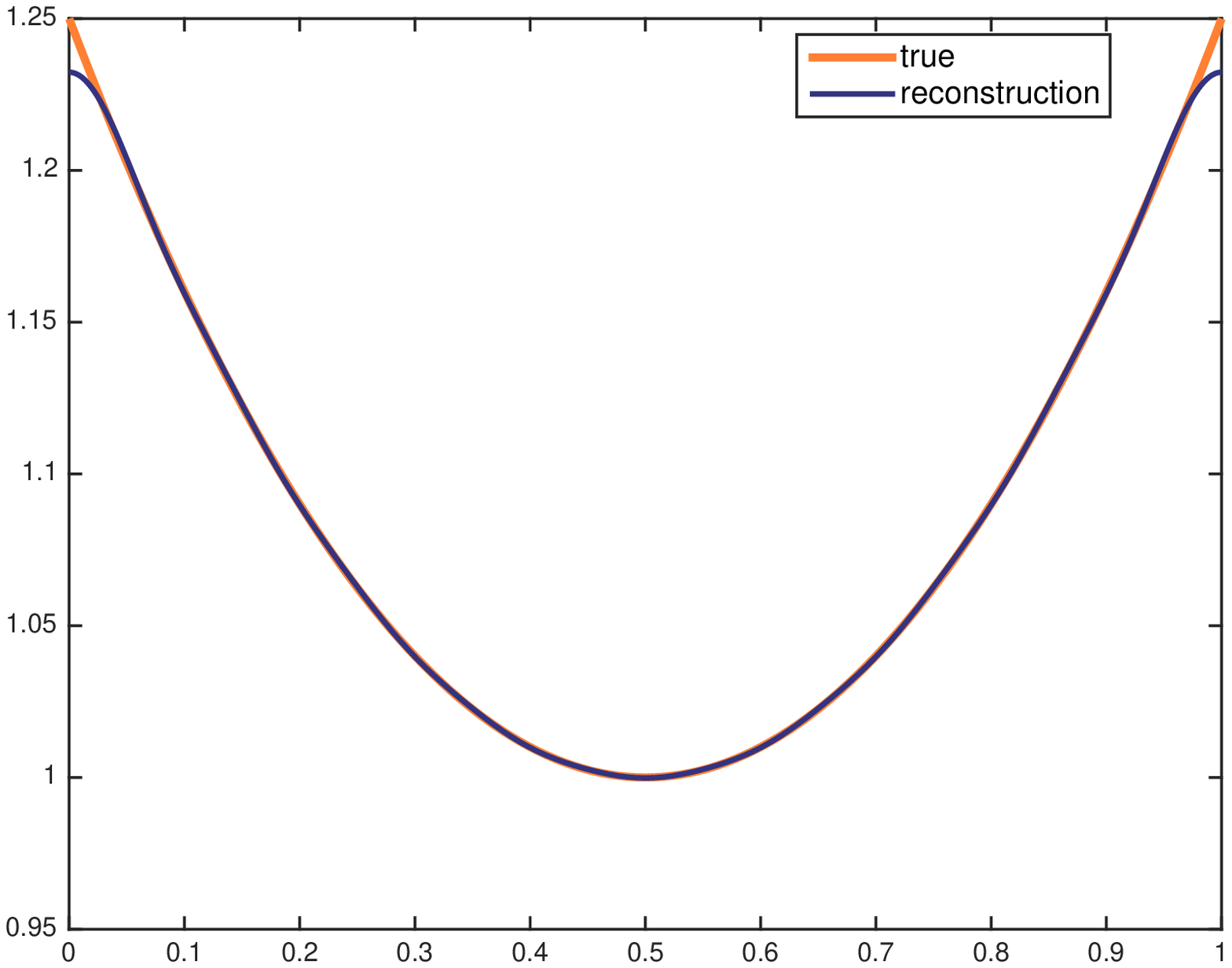}
      \caption{recovery using $10$ eigenvalues}
    \end{subfigure}
    \begin{subfigure}[t]{0.42\textwidth}
      \includegraphics[width=\textwidth]{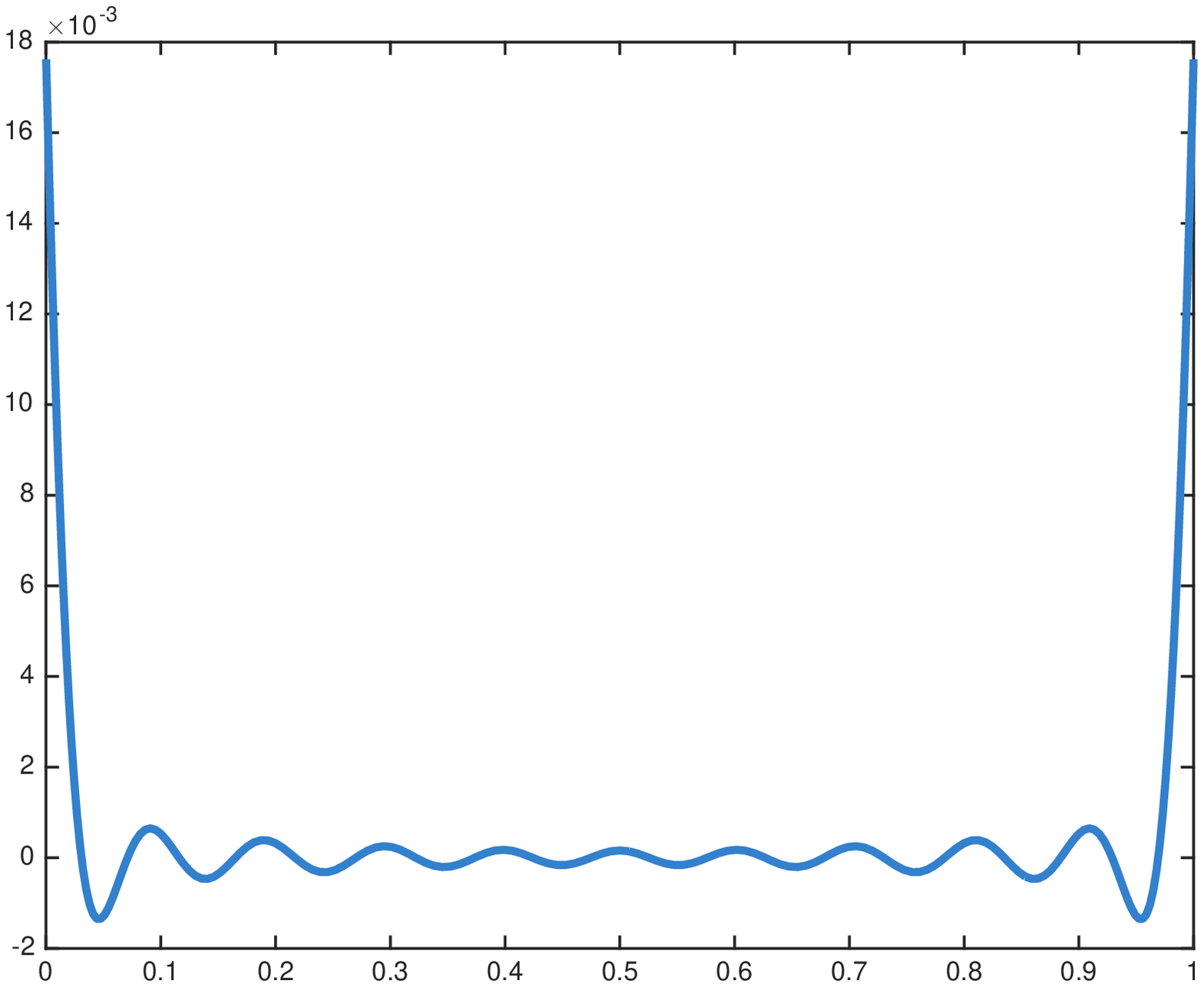}
      \caption{mismatch using $10$ eigenvalues}
    \end{subfigure}
       \begin{subfigure}[t]{0.42\textwidth}
      \includegraphics[width=\textwidth]{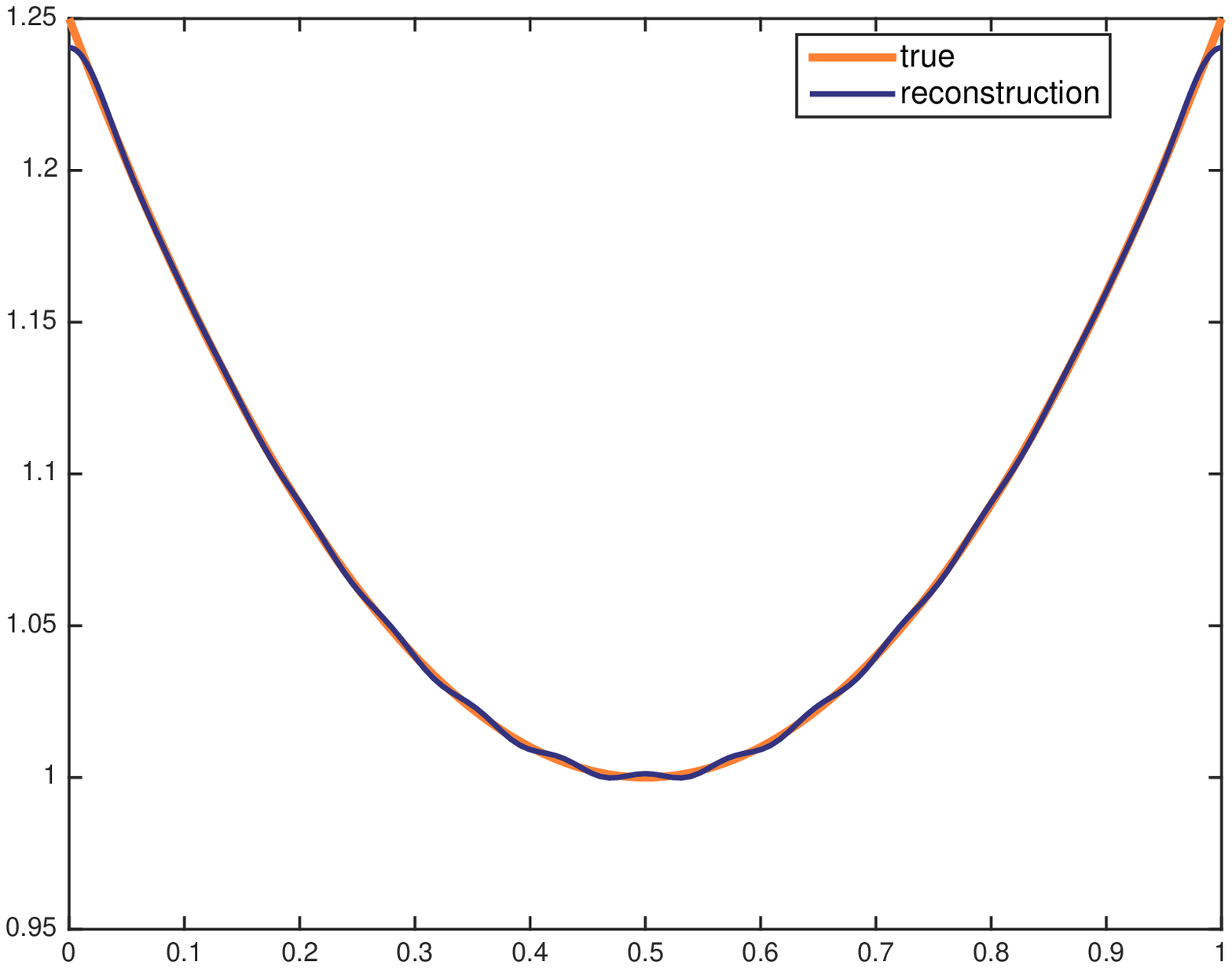}
      \caption{recovery using $15$ eigenvalues}
    \end{subfigure}
    \begin{subfigure}[t]{0.42\textwidth}
      \includegraphics[width=\textwidth]{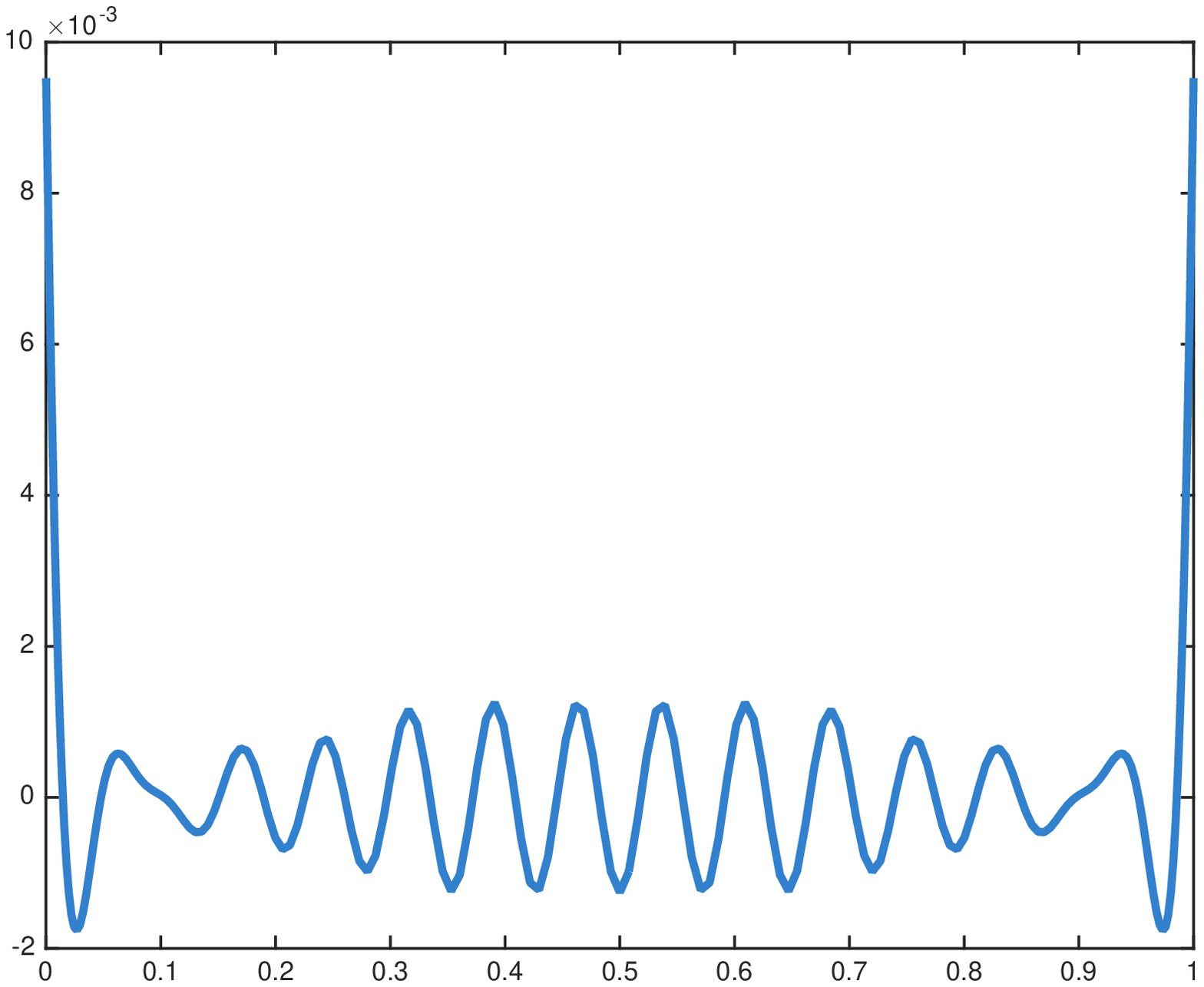}
      \caption{mismatch using $15$ eigenvalues}
    \end{subfigure}
    \caption{Recovery of $\rho_2(x)$ using $5,\, 10,$ and $15$ eigenvalues.}\label{E2}
  \end{figure}

\begin{figure}
           \begin{subfigure}[t]{0.42\textwidth}
      \includegraphics[width=\textwidth]{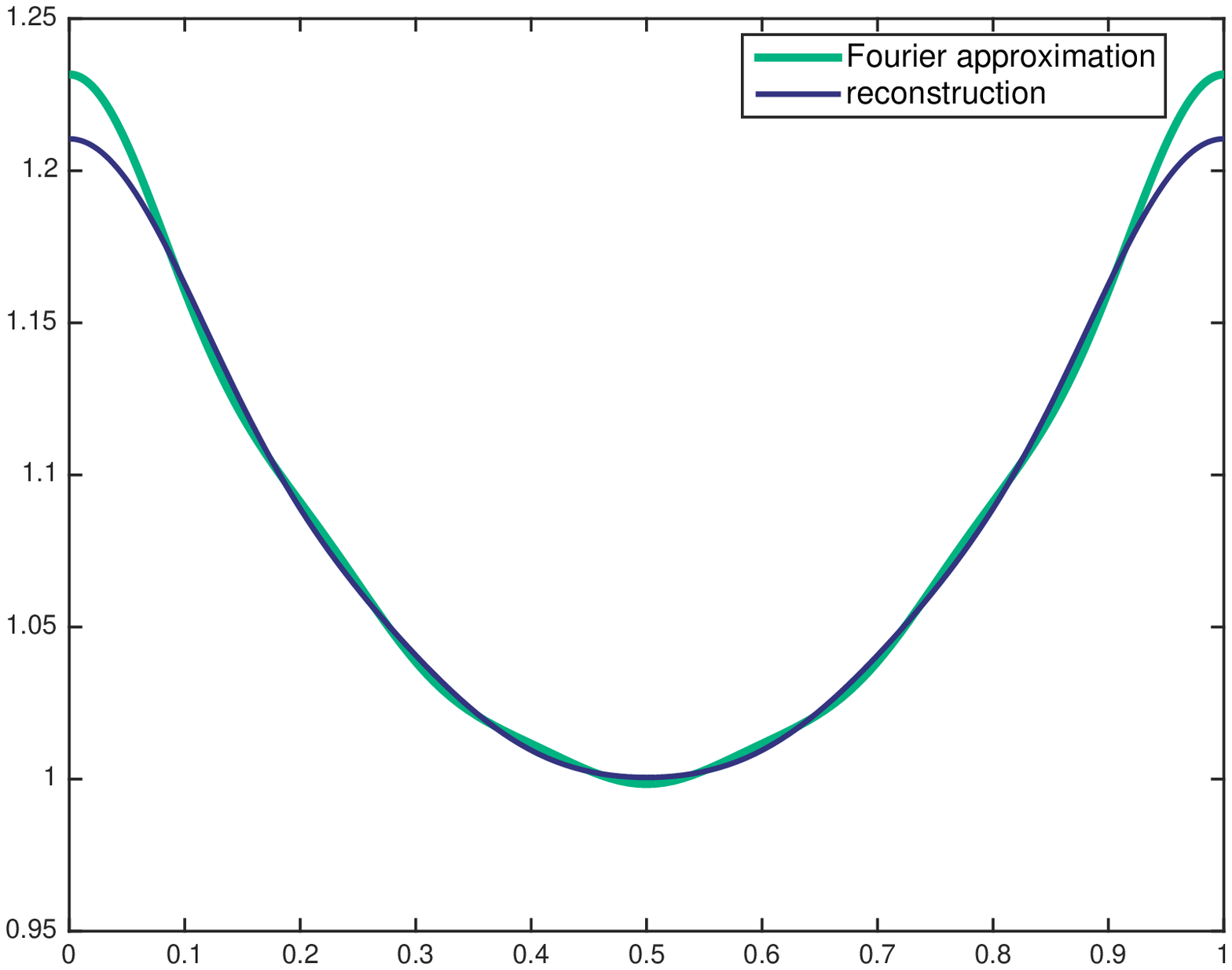}
      \caption{reconstruction and Fourier approximation with $M=K=5$}
    \end{subfigure}
    \begin{subfigure}[t]{0.42\textwidth}
      \includegraphics[width=\textwidth]{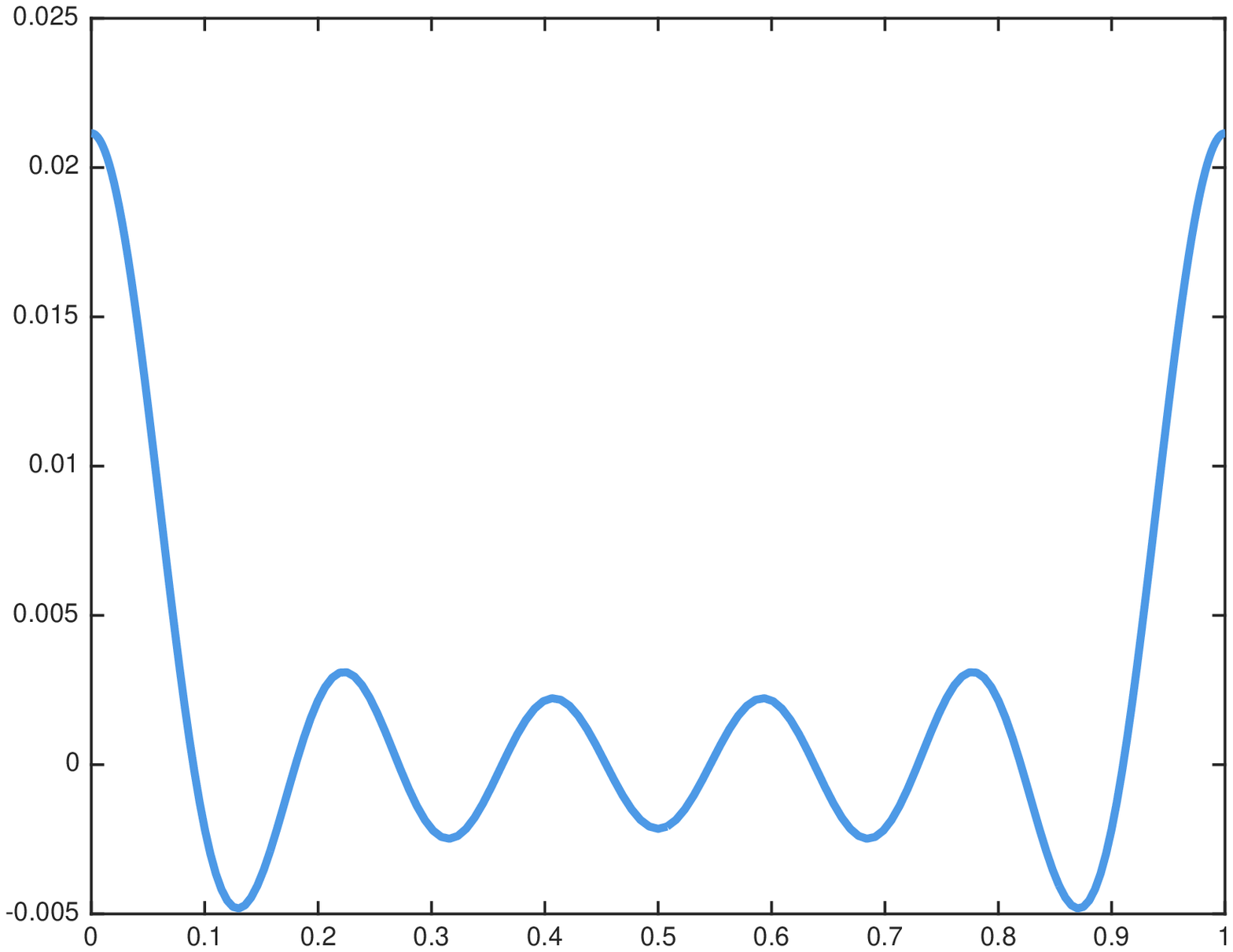}
      \caption{their difference}
    \end{subfigure}
       \begin{subfigure}[t]{0.42\textwidth}
      \includegraphics[width=\textwidth]{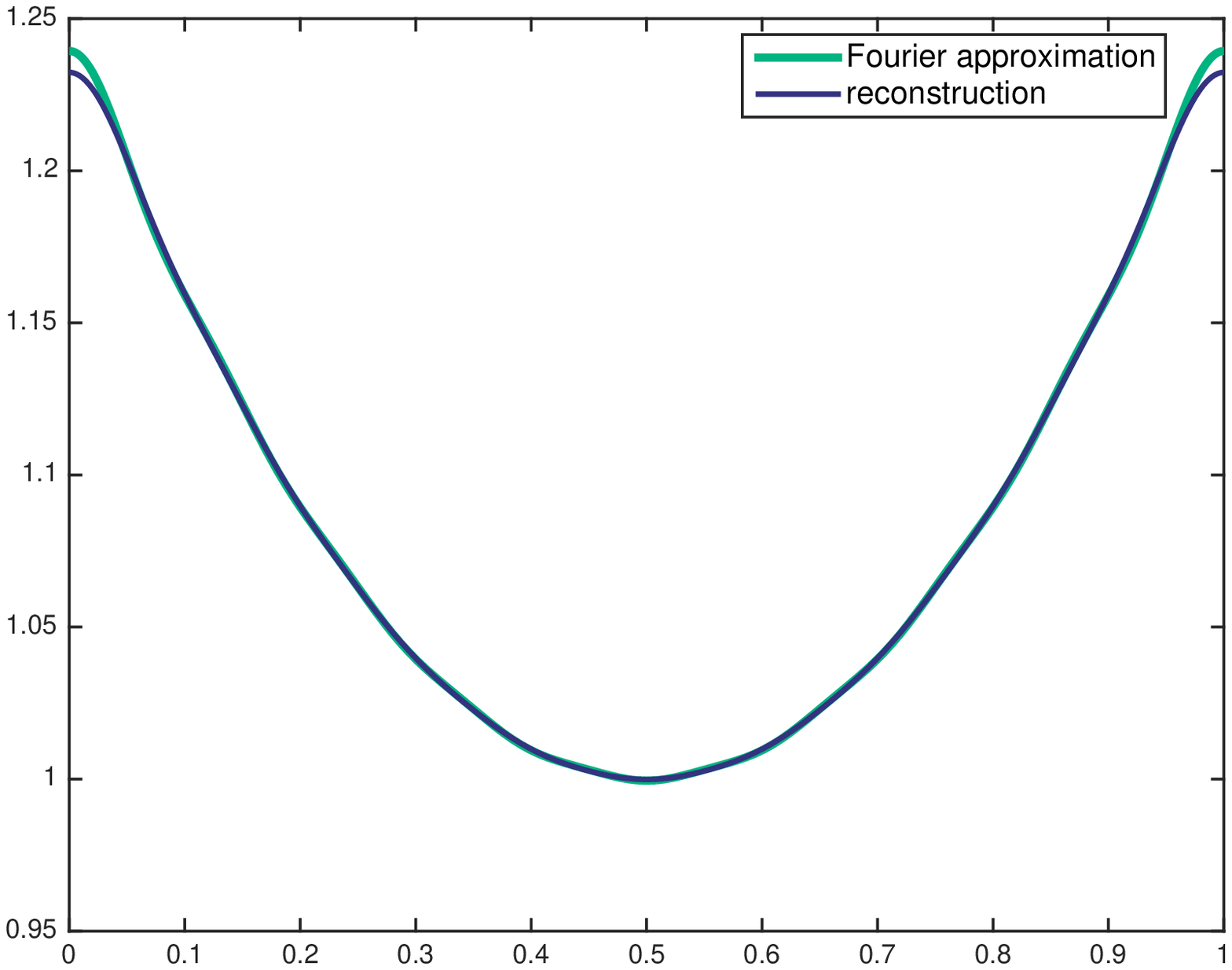}
      \caption{reconstruction and Fourier approximation with $M=K=10$}
    \end{subfigure}
    \begin{subfigure}[t]{0.42\textwidth}
      \includegraphics[width=\textwidth]{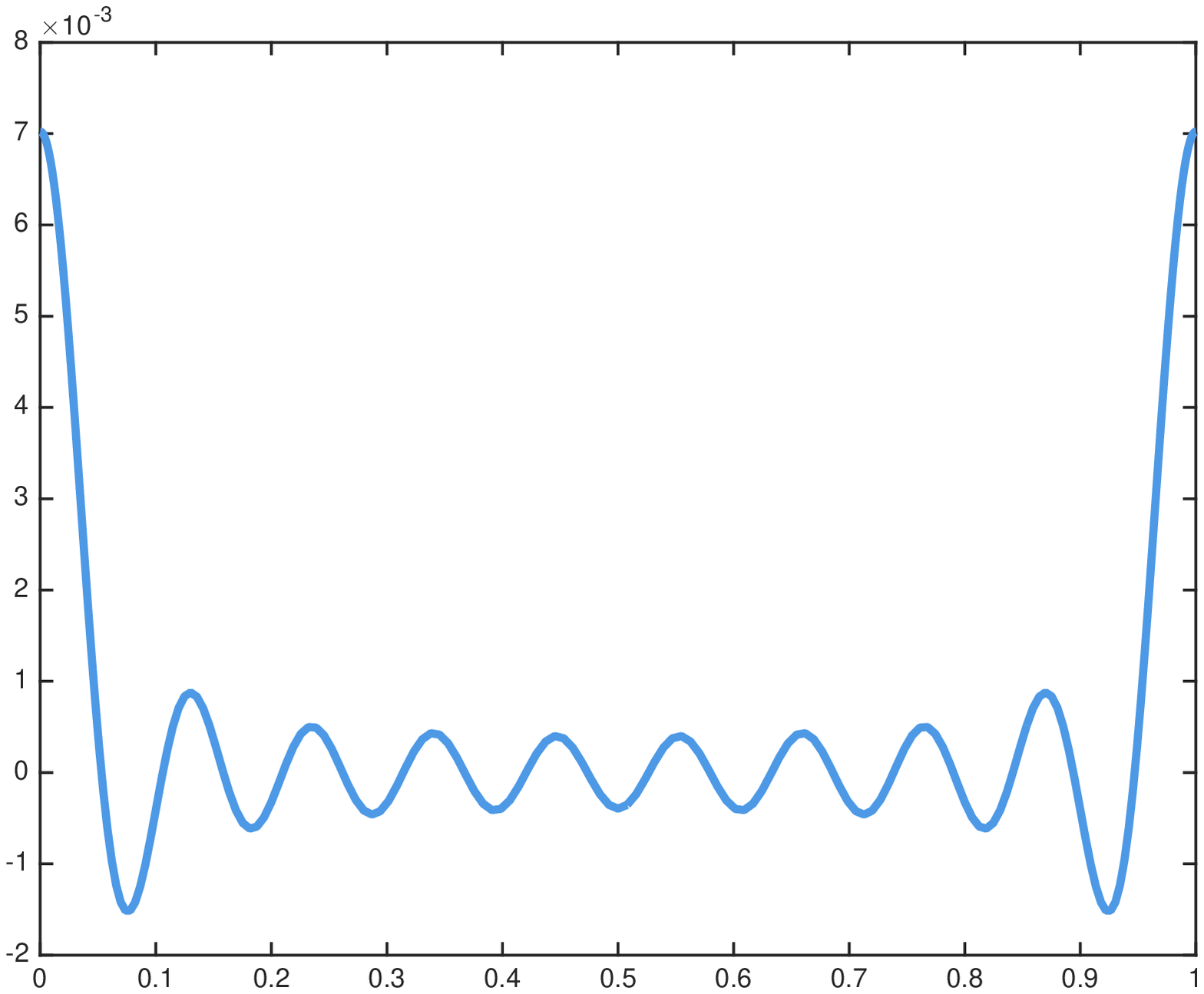}
      \caption{their difference}
    \end{subfigure}
     \begin{subfigure}[t]{0.42\textwidth}
      \includegraphics[width=\textwidth]{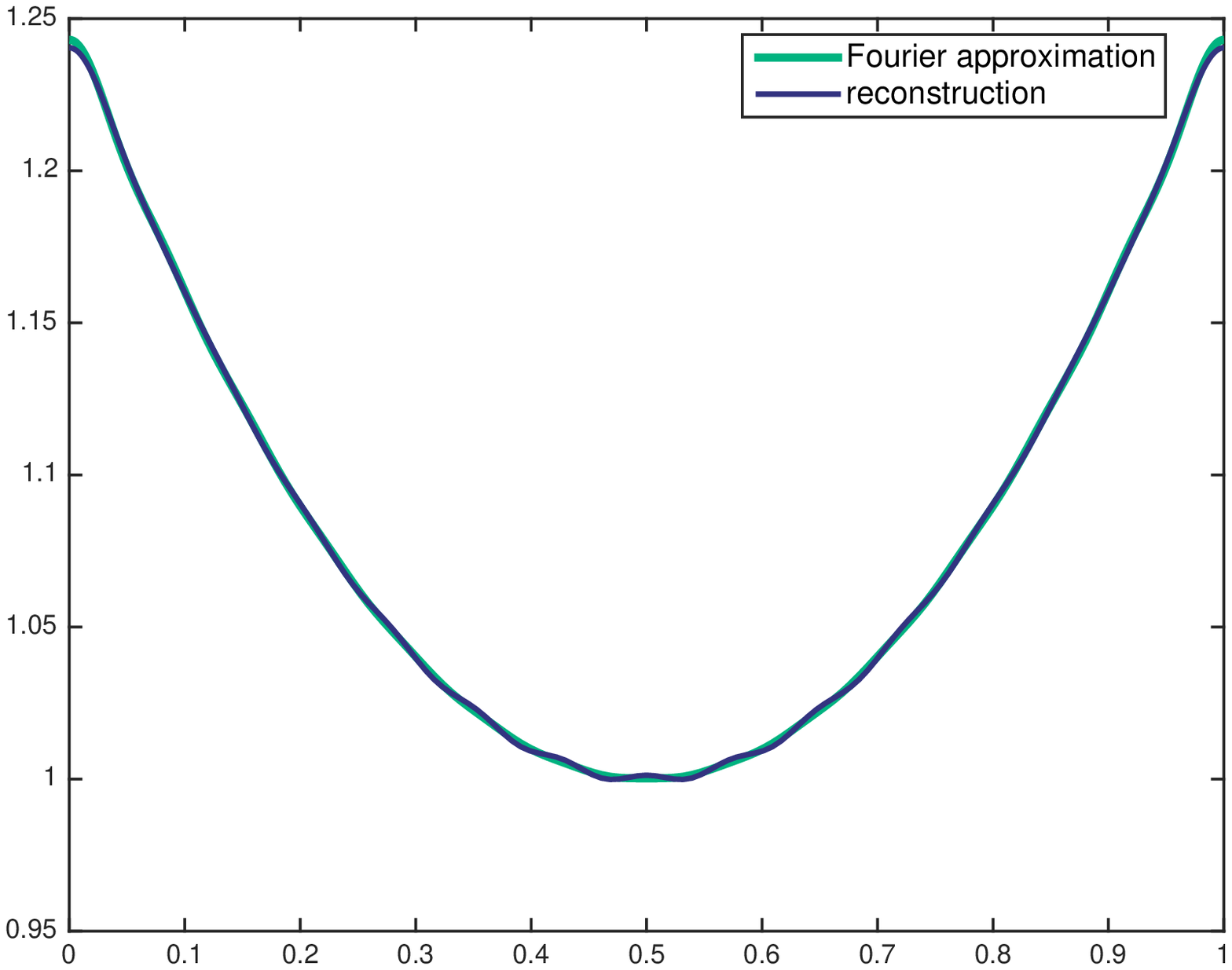}
      \caption{reconstruction and Fourier approximation with $M=K=15$}
    \end{subfigure}
    \begin{subfigure}[t]{0.42\textwidth}
      \includegraphics[width=\textwidth]{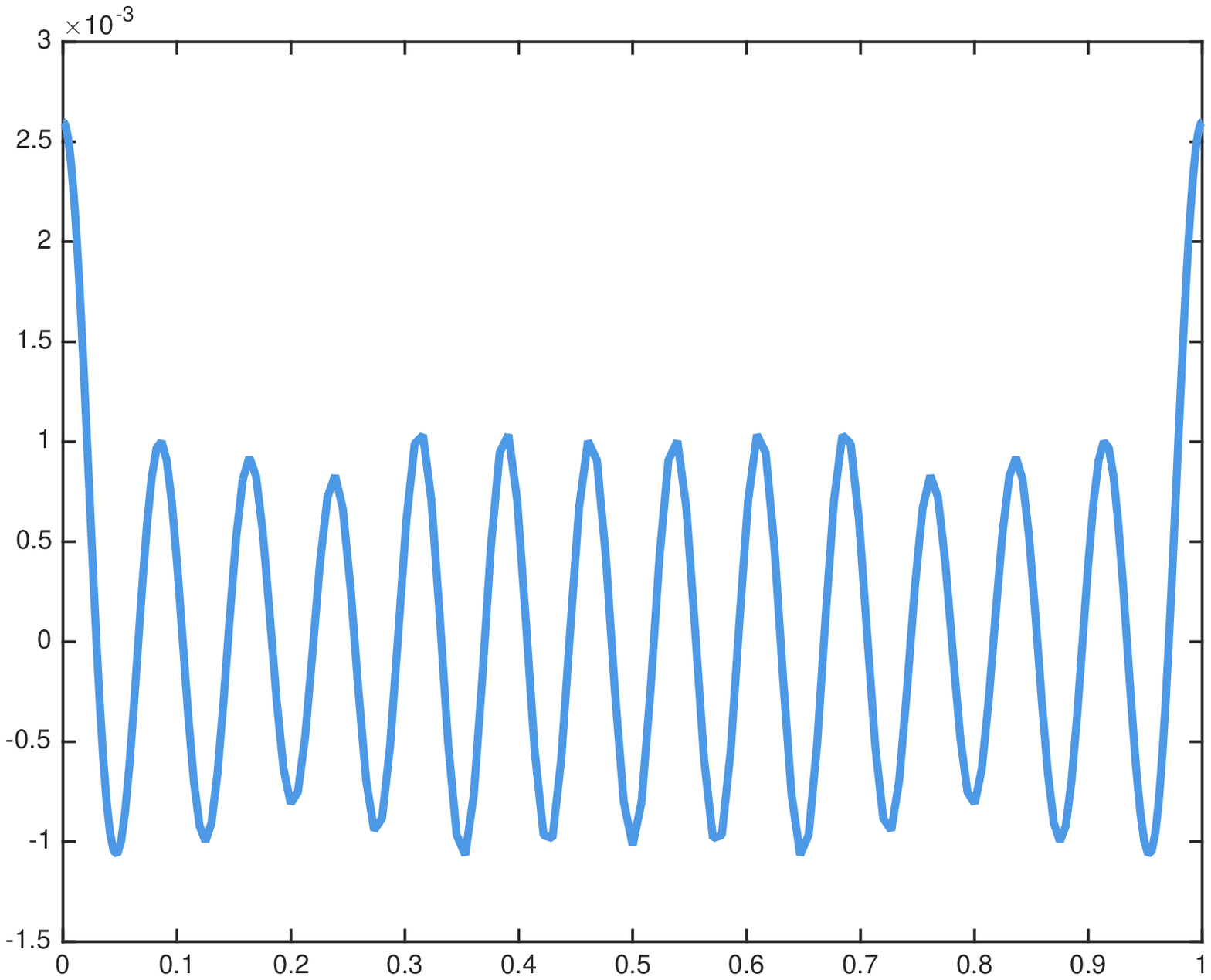}
      \caption{their difference}
    \end{subfigure}
    \caption{Recovery of $\rho_2(x)$ and its difference with the Fourier approximation}\label{E3}
  \end{figure}

\subsection{Example 3}
We will recover
\[
\rho_3(x)=1-0.3e^{-20(x-0.5)^2}
\]
using $K=7$ eigenvalues with noises. We set $N=300$ and $J=15$. As is discussed in \cite{lowe1992recovery}, we can never expect to do the recovery with even a few percent of errors in the eigenvalues themselves. And for the standard inverse Sturm-Liouville problem $(\ref{eq11})$, the appropriate measure of accuracy of the data are the numbers $c_k=\lambda_k-k^2\pi^2-\int_0^1q(x)\mathrm{d}x$.
For our experiment, we added to each eigenvalue $\lambda_k$ some noise that is normally distributed with zero mean and standard deviation $0.05$ and $0.1$. Figure \ref{E33} shows the reconstruction.

\begin{figure}
           \begin{subfigure}[t]{0.42\textwidth}
      \includegraphics[width=\textwidth]{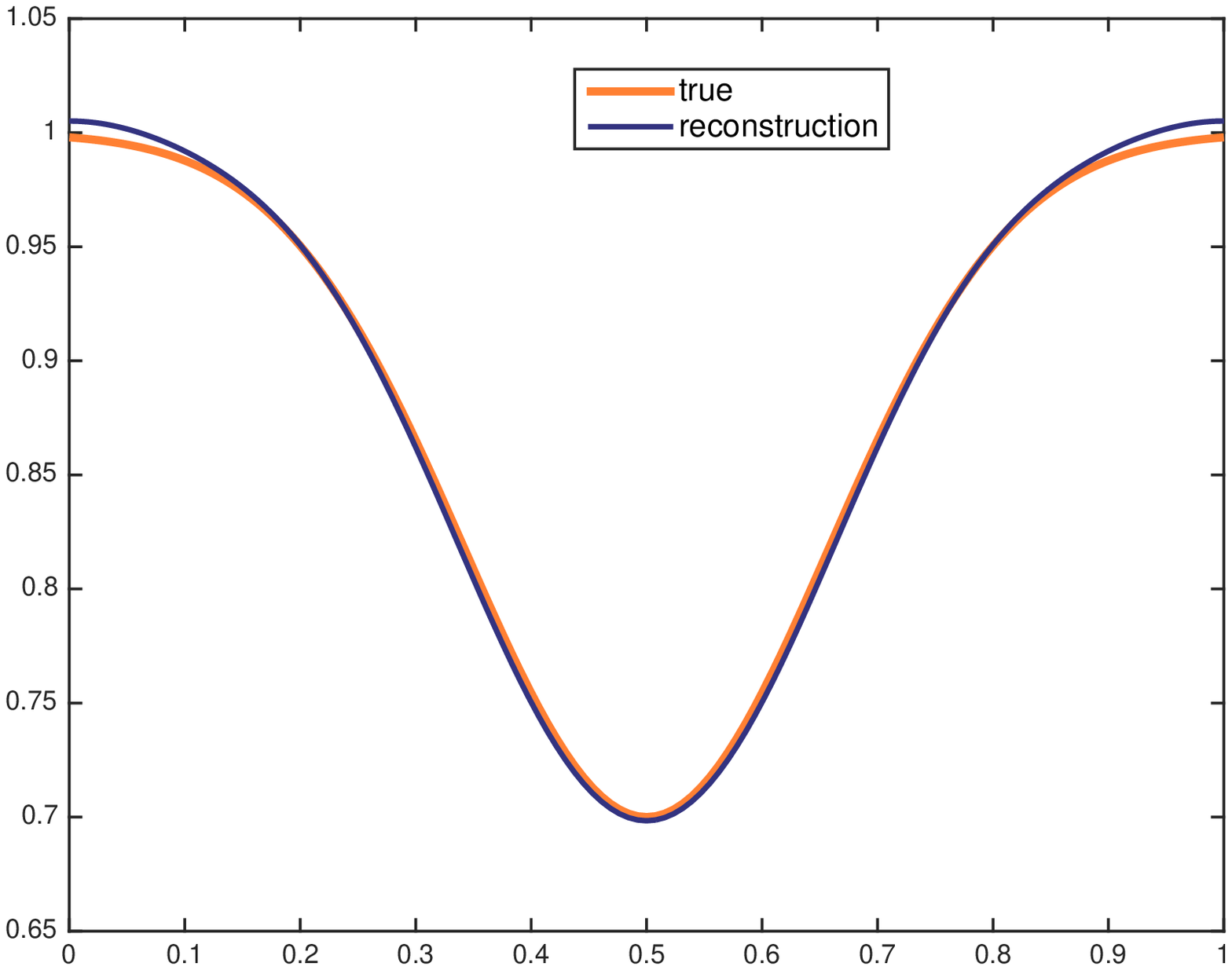}
      \caption{recovery: noise standard dev = 0.05}
    \end{subfigure}
    \begin{subfigure}[t]{0.42\textwidth}
      \includegraphics[width=\textwidth]{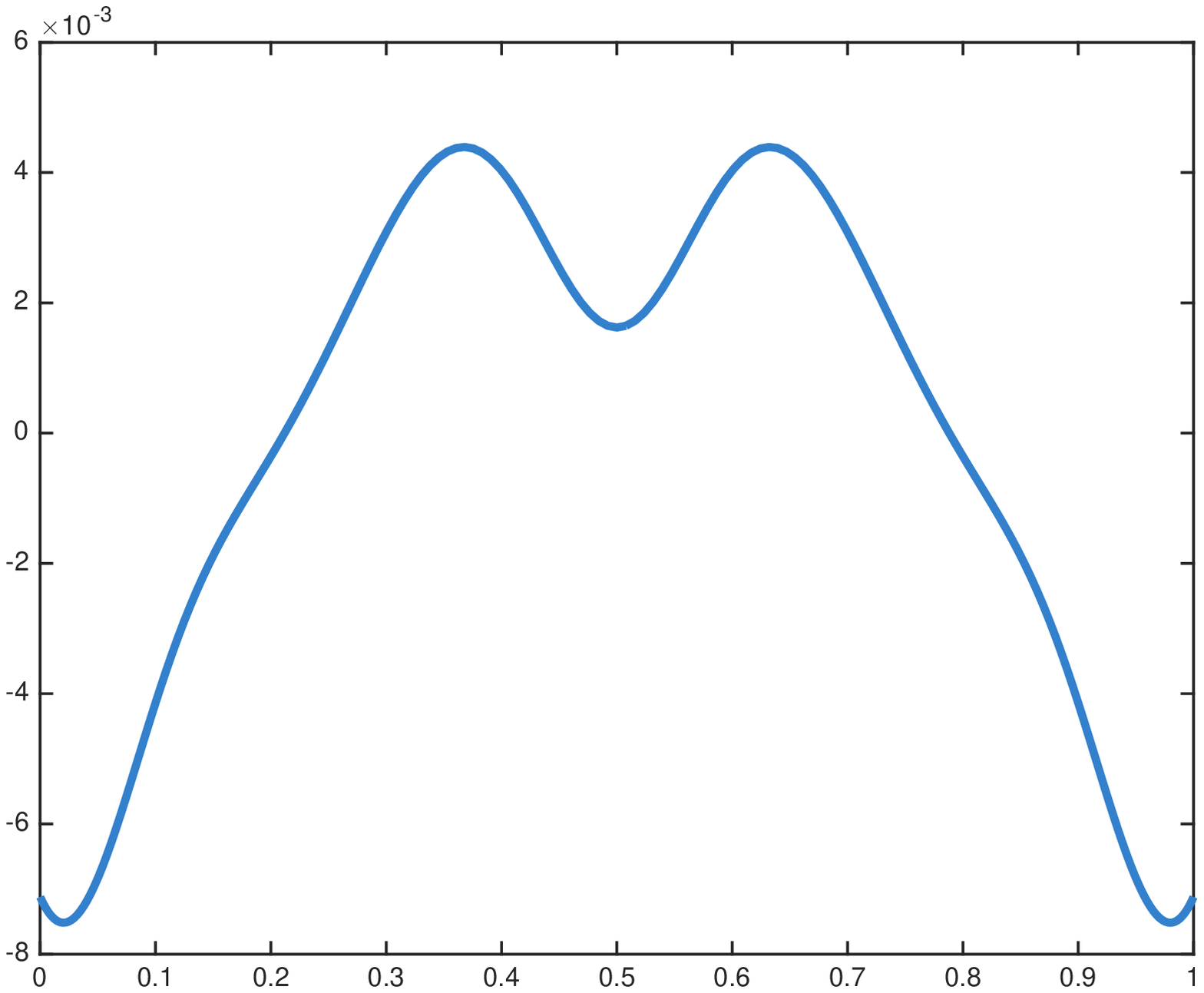}
      \caption{mismatch: noise standard dev = 0.05}
    \end{subfigure}
       \begin{subfigure}[t]{0.42\textwidth}
      \includegraphics[width=\textwidth]{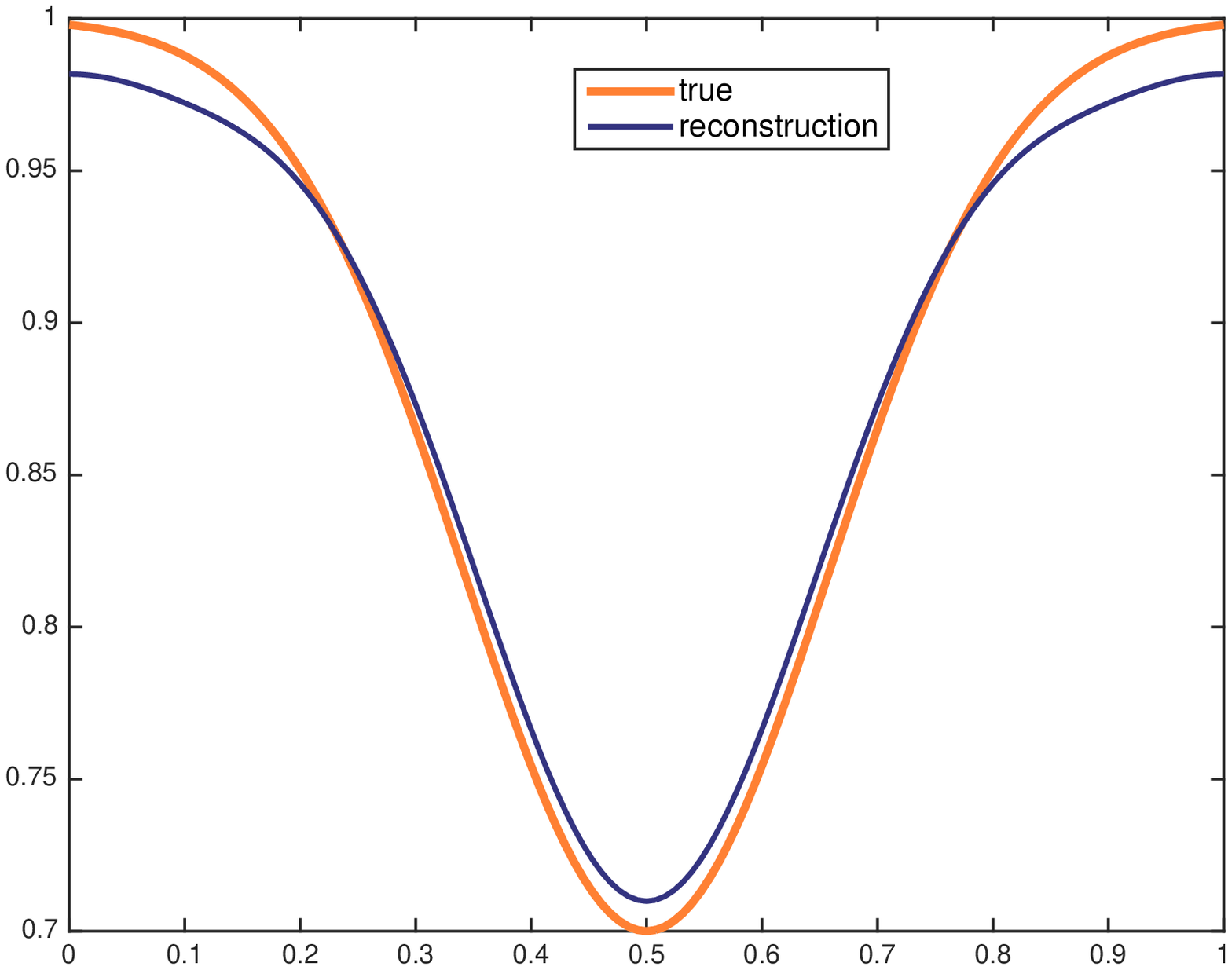}
      \caption{recovery: noise standard dev = 0.1}
    \end{subfigure}
    \begin{subfigure}[t]{0.42\textwidth}
      \includegraphics[width=\textwidth]{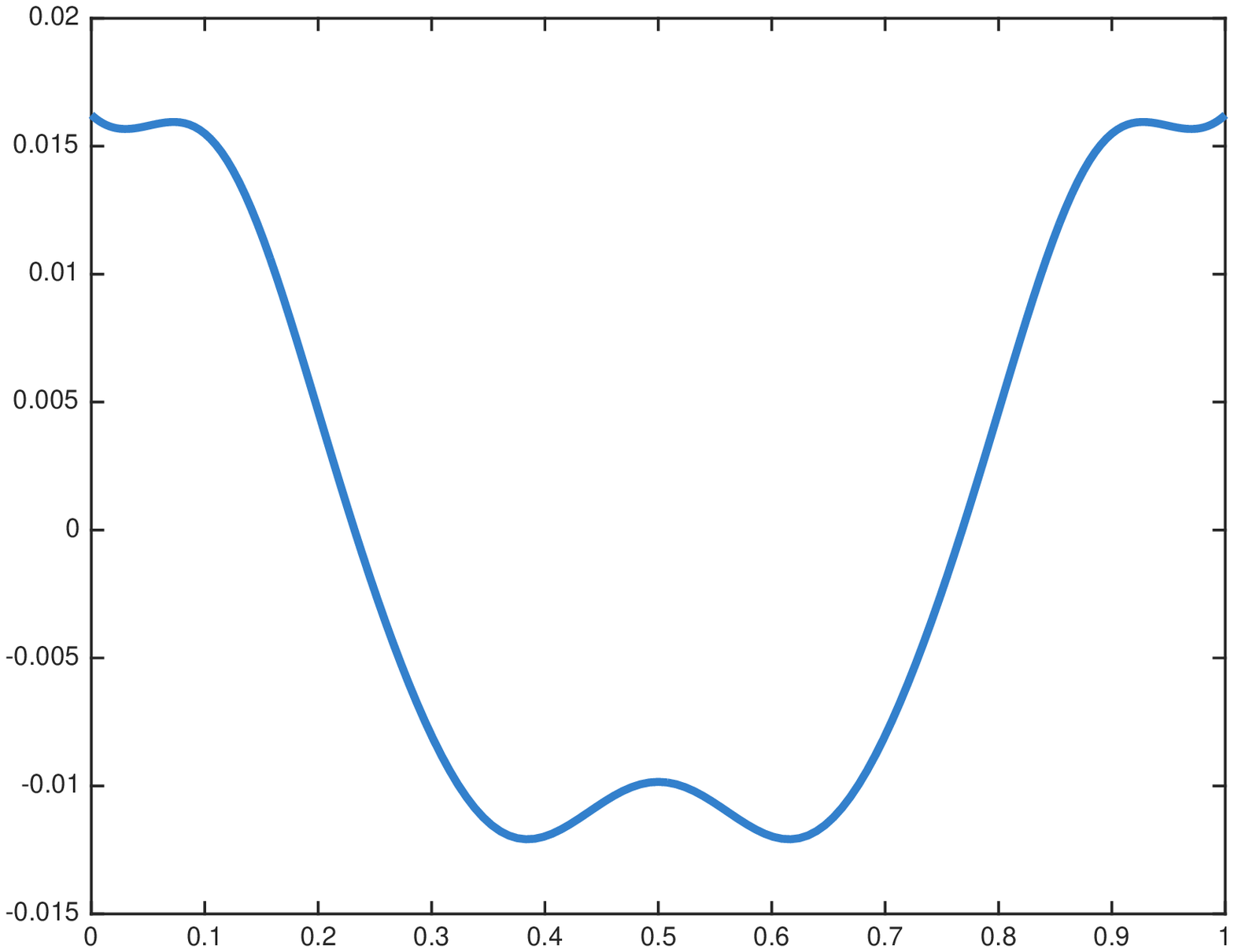}
      \caption{mismatch: noise standard dev = 0.1}
    \end{subfigure}
    \caption{Recovery of $\rho_3(x)$ with noises}\label{E33}
  \end{figure}
\subsection{Example 4}
For the fourth example, we work with a discontinuous function
\[
\rho_4(x)=\begin{cases}1,\quad &0<x<0.3,\\
1.1,\quad &0.3<x<0.7,\\
1,\quad &0.7<x<1.\end{cases}
\]
We first use $L=9,\, 15$ eigenvalues and reconstruct with $M=K$ basis functions. The take $N=1000$ and $J=30$. The results are shown in Figure \ref{E4}. The approximation of this discontinuous functions is not so close its Fourier approximation, but one can see from the Table \ref{T2} that the eigenvalues for the reconstructed density are actually closer to the ``exact" eigenvalues than the Fourier approximation.

Then we recover, with $M=9,\, 15$ basis functions, from $K=15,\,20$ eigenvalues. Still we take $N=1000$ and $J=30$. The results are also shown in Figure \ref{E4}.

\begin{table}[h]
  \begin{tabular}{|c | c | c |c| c| c| c|c|}
\hline
 &  &  \multicolumn{2}{m{3.3cm}|}{reconstruction with $K=M=9$}&  \multicolumn{2}{m{3.3cm}|}{reconstruction with $K=M=15$} &  \multicolumn{2}{m{3.8cm}|}{Fourier approximation with $15$ basis functions}\\
    \hline
   & true &  eigenvalue & mismatch&  eigenvalue & mismatch & eigenvalue & mismatch\\ \hline
$\lambda_1$&9.2151	&9.2151	&0.0000&	9.2151	&0.0000 &9.2200&	-0.0049\\
$\lambda_2$&38.2571&	38.2571	&0.0000	&38.2571&	0.0000&38.2867&-0.0296\\
$\lambda_3$&86.0274	&86.0274	&0.0000	&86.0274&	0.0000&86.0367&-0.0092\\
$\lambda_4$&150.8596&	150.8596	&0.0000	&150.8596&	0.0000&150.8942&-0.0346\\
$\lambda_5$&236.9245&	236.9248	&-0.0002&	236.9245&	0.0001&237.1215&-0.1969\\
$\lambda_6$&343.4206&	343.4212	&-0.0006	&343.4206	&0.0000&343.5498&	-0.1292\\
$\lambda_7$&464.5170&	464.5183	&-0.0013&	464.5163	&0.0007&464.5289	&-0.0119\\
$\lambda_8$&605.4567&	605.4922	&-0.0354&	605.4575	&-0.0008&605.8644&	-0.4077\\
$\lambda_9$&770.6854	&770.2775	&0.4079&	770.6869	&-0.0015&771.2058&	-0.5203\\
$\lambda_{10}$&9.5050E+02			&---------&---------&9.5050E+02	&0.0073&9.5050E+02&	-0.0177\\
$\lambda_{11}$&1.1453E+03			&---------&---------&1.1453E+03	&-0.0133&1.1458E+03&	-0.4331\\
$\lambda_{12}$&1.3668E+03			&---------&---------&1.3668E+03	&0.0185&1.368E+03	&-1.1854\\
$\lambda_{13}$&1.6077E+03			&---------&---------&1.6076E+03	&0.0083&1.608E+03	&-0.3431\\
$\lambda_{14}$&1.8581E+03			&---------&---------&1.8582E+03	&-0.0893&1.8581E+03&	-0.0233\\
$\lambda_{15}$&2.1325E+03			&---------&---------&2.1321E+03	&0.4194&2.1352E+03&	-2.6319\\
  \hline
  \end{tabular}
  \caption{True eigenvalues \textit{vs} eigenvalues for reconstructed $\rho_4(x)$ and the Fourier approximation}\label{T2}
\end{table}

\begin{figure}
    \begin{subfigure}[t]{0.49\textwidth}
      \includegraphics[width=\textwidth]{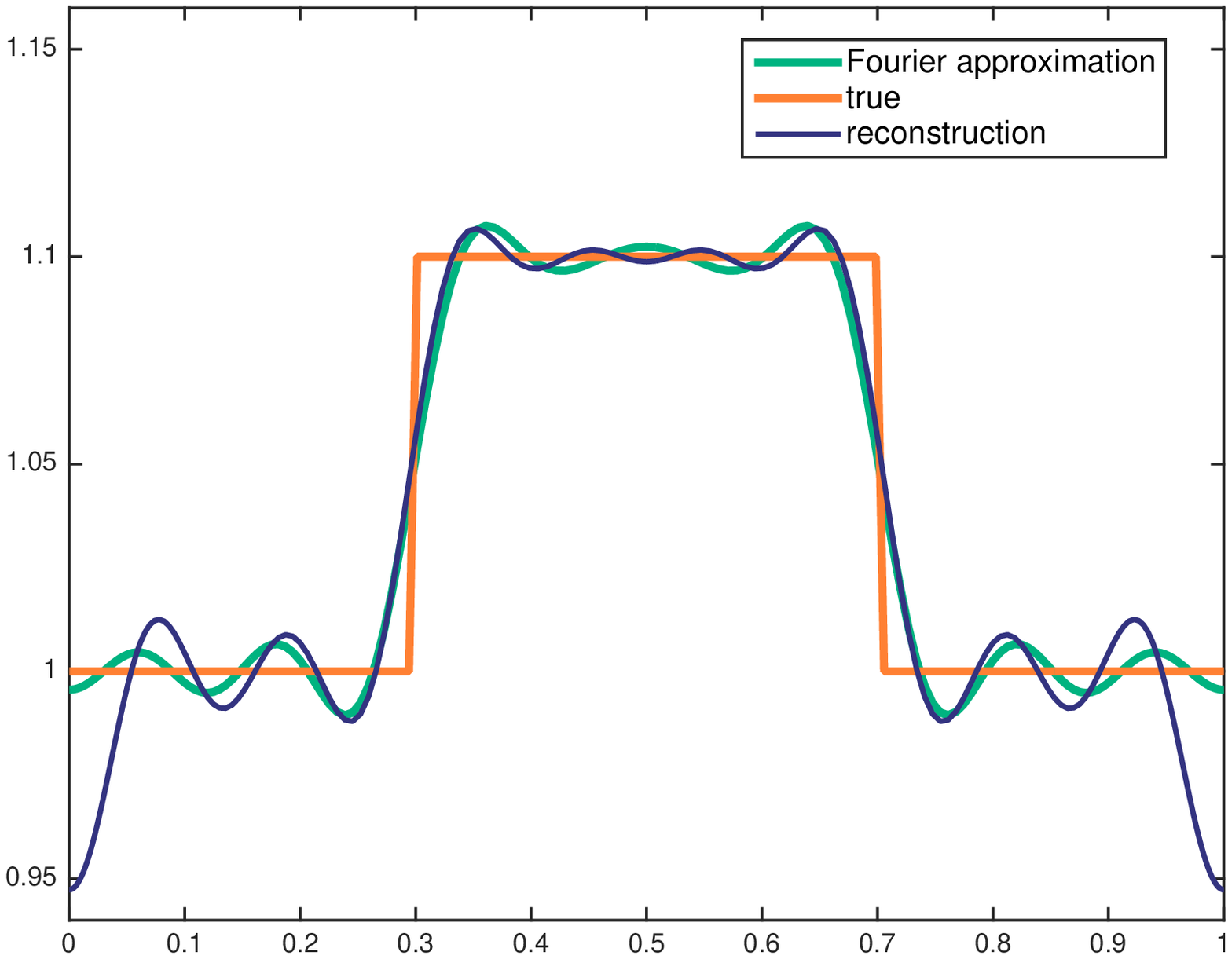}
      \caption{recovery with $K=9$, $M=9$}
    \end{subfigure}
        \begin{subfigure}[t]{0.49\textwidth}
      \includegraphics[width=\textwidth]{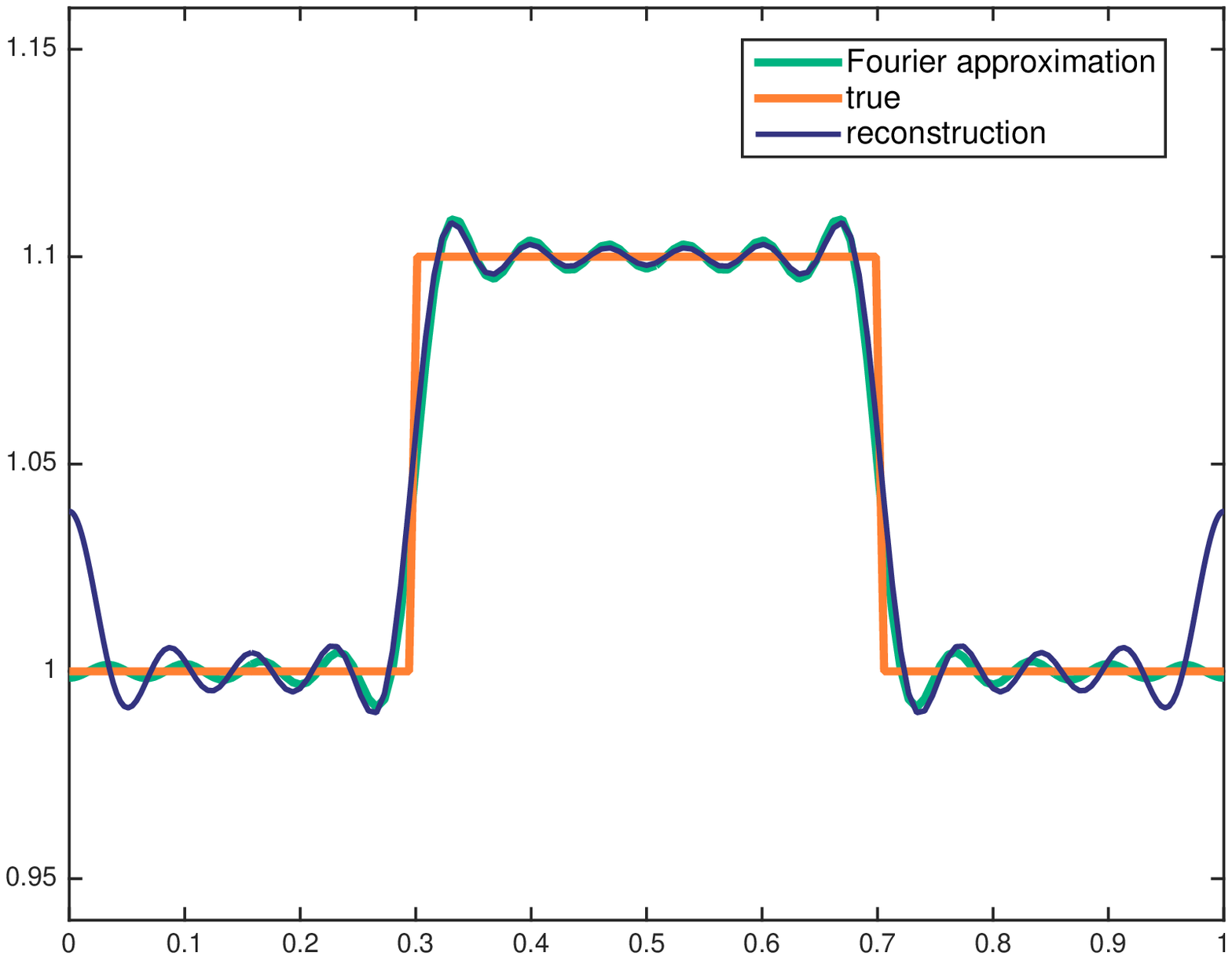}
      \caption{recovery with $K=15$, $M=15$}
    \end{subfigure}
    \begin{subfigure}[t]{0.49\textwidth}
      \includegraphics[width=\textwidth]{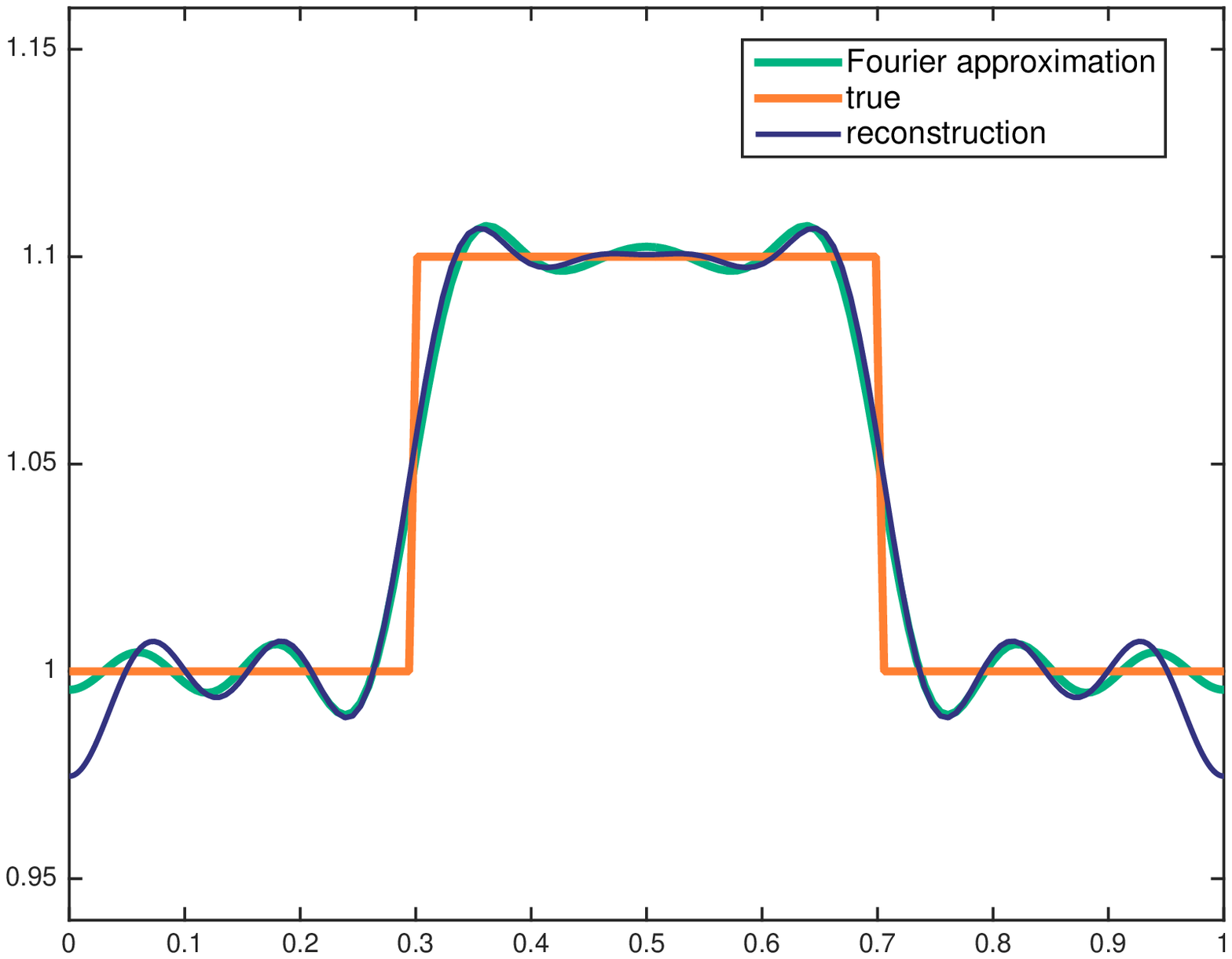}
      \caption{recovery with $K=15$, $M=9$}
    \end{subfigure}
        \begin{subfigure}[t]{0.49\textwidth}
      \includegraphics[width=\textwidth]{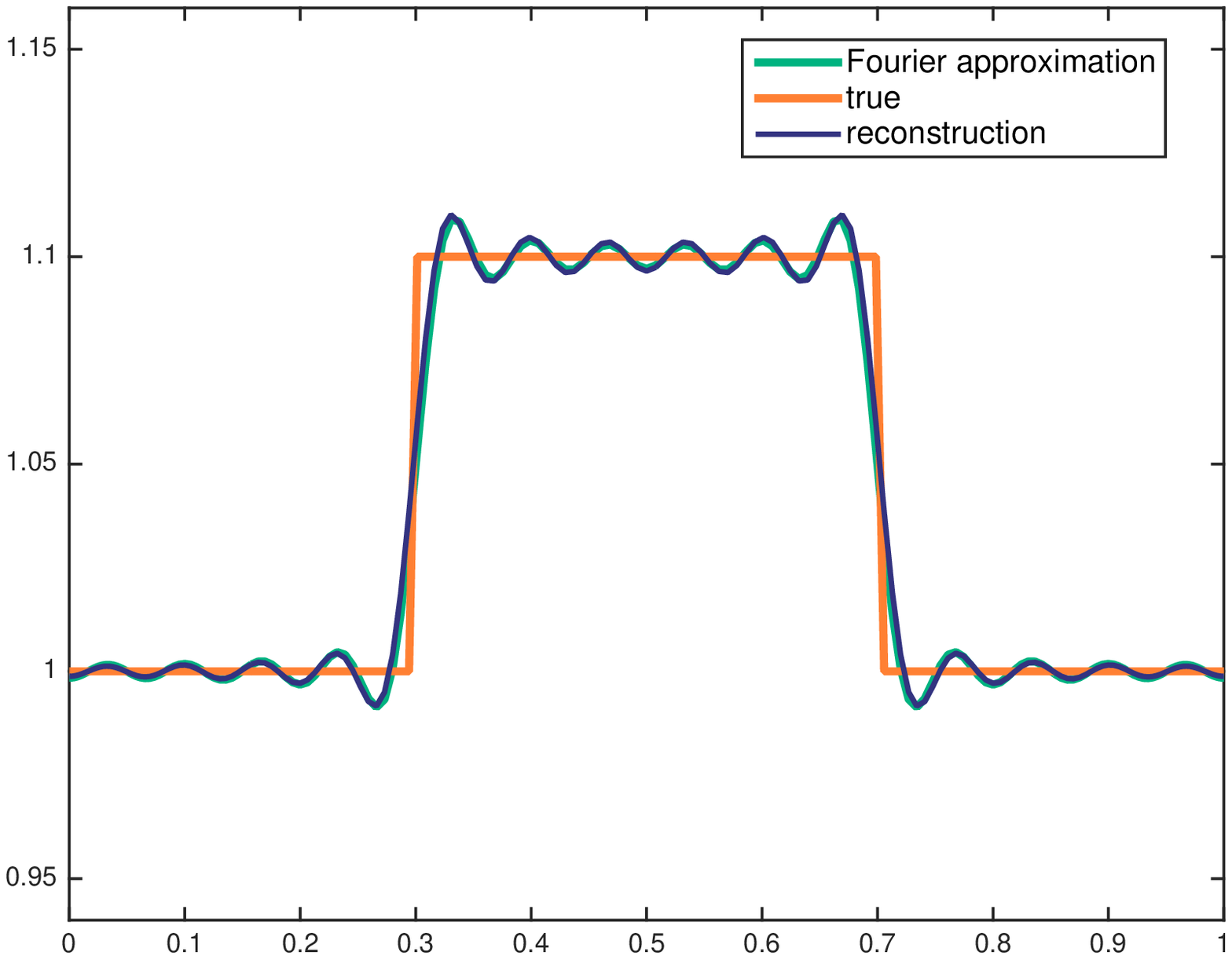}
      \caption{recovery with $K=20$, $M=15$}
    \end{subfigure}
    \caption{Recovery of $\rho_4(x)$ using $9$ and $15$ basis functions.}\label{E4}
  \end{figure}

\section{Discussion}\label{discussion}
The main purpose of this article is not to present an algorithm cooked to perfection, but rather to propose a general strategy for solving certain types of inverse spectral problems - that is to invert a sequence of trace formulas. We believe the same idea can be adopted to solve some other inverse spectral problems. In a subsequent paper, we will give a scheme for the inverse spectral problem for the (non self-adjoint) damped wave operator
\begin{equation}\label{damp_OP}
\left(
\begin{array}{cc}
0 & I\\
\frac{\mathrm{d}^2}{\mathrm{d}x^2}&-2a(x)
\end{array}
\right)
\end{equation}
using the same strategy.\\

We believe that the algorithm can be improved in various places. For example, 1) We de not know whether some polynomials other than Chebyshev would work better. Even if we are to use Chebyshev polynomials, our choice of Chebyshev polynomials are far from being ``optimal", and there must be plenty of techniques that can be adopted to reduce the computational cost, improve the accuracy, etc. We know for sure that the Chebyshev polynomials used in above numerical experiments are way more than necessary. The algorithm worked equally well with much few Chebyshev polynomials, for example, using only $\widetilde{T}_1,\,\widetilde{T}_{16},\,\widetilde{T}_{31}\cdots$ for \textit{Example 4}; 2) for the representation of $\rho$, one can choose different bases other than trigonometric functions $(\ref{cosine_expansion})$, if the function $\rho$ can not be well approximated by Fourier cosine functions. This also shows the great flexibility and potential of the algorithm.\\

 A practical inverse spectral problem usually involves very limited data (the low lying eigenvalues that are measurable), so one can not expect to recover too many unknowns. Numerically, we are inverting some map
 \[
 x\mapsto \mathcal{F}(x),
 \]
 where the dimension of $x$ and $\mathcal{F}(x)$ are both (at least intrinsically) quite small. However, the relation between the parameter to be recovered and the data is quite complicated. The amount of operations needed to compute $\mathcal{F}(x)$ might be astronomical, and vary in magnitude for different forward maps. The method presented in this article suggests a choice of the forward map $\mathcal{F}$, which only requires basic matrix operations.\\


\noindent\textbf{Other boundary conditions.}
First consider the inverse spectral problem for
\begin{equation}\label{eq2}
\begin{split}
&-\frac{\mathrm{d}^2u}{\mathrm{d}x^2}=\lambda\rho u,\quad\quad\quad x\in[0,1],\\
&u(0)=u'(1)=0.
\end{split}
\end{equation}
The eigenvalues for the above problem, together with the Dirichlet eigenvalues, should be sufficient to determine a general $\rho$. Now, denote $A\in\mathcal{L}(L^2(0,1),L^2(0,1))$ to be the operator defined as $v=Af$, where $v$ satisfies
\begin{equation}
\begin{split}
&-\frac{\mathrm{d}^2v}{\mathrm{d}x^2}=f,\quad\quad\quad x\in[0,1],\\
&v(0)=v'(1)=0.
\end{split}
\end{equation}
The Schwartz kernel associated with operator $A$ is
\[
g(x,y)=\begin{cases}x,\quad\quad 0\leq x\leq y\leq 1,\\
y,\quad\quad 0\leq y\leq x\leq 1.
\end{cases}
\]
We can use the eigenvalues and eigenfunctions $\{\mu_n,\phi_n\}_{n=1}^\infty$ of $-\Delta$ with the new boundary conditions,
\[
\mu_n=\left(\frac{2n-1}{2}\right)^2\pi^2,\quad\phi_n(x)=\sqrt{2}\sin \frac{2n-1}{2}\pi x,
\]
to represent
\begin{equation}\label{greenrepresentation}
g(x,y)=\sum_{n=1}^\infty\mu_n^{-1}\phi_n(x)\phi_n(y).
\end{equation}
As long as the boundary value problem is well-posed with imposed boundary conditions, one can derive a representation of Green's function for the Laplacian like $(\ref{greenrepresentation})$. However, it seems that the method does not work for Neumann problem
\begin{equation}\label{eq3}
\begin{split}
&-\frac{\mathrm{d}^2u}{\mathrm{d}^2x}=\lambda\rho u,\quad\quad\quad x\in[0,1],\\
&u'(0)=u'(1)=0.
\end{split}
\end{equation}
There is always an eigenvalue $0$, and the boundary value problem
\begin{equation}
\begin{split}
&-\frac{\mathrm{d}^2v}{\mathrm{d}^2x}=\rho f,\quad\quad\quad x\in[0,1],\\
&v'(0)=v'(1)=0
\end{split}
\end{equation}
is not well-posed. 

~\\
\noindent\textbf{Acknowledgements.} This project started when JZ attended an advanced topics course offered by Steven Cox at Rice University during Fall 2015. The authors are indebted to Steven Cox's first idea of trace inversion and stimulating discussions. JZ also gratefully acknowledges the helpful discussion with Xiao Liu and Peter Caday. XX was partly supported by NSFC grant No. 11471284, 11421110002, 116221101, 91630309 and the Fundamental Research Funds for the Central Universities.

\bibliographystyle{abbrv}
\bibliography{biblio}

\end{document}